\documentclass[11pt]{amsart}

\usepackage{amssymb,amscd,amsmath,hyperref,color,enumerate,tikz-cd}
\usepackage[all,cmtip]{xy}

\usepackage[margin=3cm]{geometry}

\newcommand{\R}{{\ensuremath{\mathbb{R}}}}
\newcommand{\N}{{\ensuremath{\mathbb{N}}}}

\newcommand{\C}{{\ensuremath{\mathbb{C}}}}

\def\U{\mathcal{U}}

\def\lm{\lambda}

\def\ep{\epsilon}

\def\EU{\mathrm{EUCP}}

\def\dx{\mathrm{Dix}}
\def\dxbar{\overline{\mathrm{Dix}}}
\def\mg{\mathrm{Mag}}
\def\mgbar{\overline{\mathrm{Mag}}}

\def\P{\mathrm{Prim}}
\def\M{\mathrm{Max}}
\def\I{\mathrm{Ideal}}
\def\G{\mathrm{Glimm}}

\def\sp{\mathrm{sp}}
\def\ds{\mathrm{dist}}

\def\spn{\mathop{\mathrm{span}}}
\def\di{\mathop{\mathrm{diag}}}
\def\ds{\mathop{\mathrm{dist}}}
\def\int{\mathop{\mathrm{Int}}}

\def\cq{\mathrm{CQ}}

\newtheorem{proposition}{Proposition}[section]
\newtheorem{lemma}[proposition]{Lemma}

\newtheorem{theorem}[proposition]{Theorem}
\newtheorem{corollary}[proposition]{Corollary}
\theoremstyle{definition}
\newtheorem{remark}[proposition]{Remark}

\newtheorem{problem}[proposition]{Problem}

\newtheorem{example}[proposition]{Example}

\numberwithin{equation}{section}

\begin{document}

\title[]{Local variants of the Dixmier property and weak centrality for C*-algebras}

%A local approach to the Dixmier property and weak centrality for C*-algebras.

\author{Robert J. Archbold}
\author{Ilja Gogi\'c}
\author{Leonel Robert}

\address{R.~J.~Archbold,  Institute of Mathematics, University of Aberdeen, King's College, Aberdeen AB24 3UE, Scotland, United Kingdom}
\email{r.archbold@abdn.ac.uk}

\address{I.~Gogi\'c, Department of Mathematics, Faculty of Science, University of Zagreb, Bijeni\v{c}ka 30, 10000 Zagreb, Croatia}
\email{ilja@math.hr}

\address{L. Robert, Department of Mathematics, University of Louisiana at Lafayette, 
Lafayette, LA 70504-1010, USA}
\email{lrobert@louisiana.edu}

%\thanks{The second named author was partially supported by the Croatian Science Foundation under the project IP-2016-06-1046.}



\date{\today}

\begin{abstract}
We study variants of the Dixmier property that apply to elements of a unital C*-algebra, rather than to the C*-algebra itself. By a  Dixmier element in a C*-algebra we understand one that can be averaged into a central element by means of a sequence of unitary mixing operators. Examples include all self-commutators and all quasinilpotent elements. We do a parallel study of an element-wise version of weak centrality, where the averaging to the centre is done using unital completely positive elementary operators (as in Magajna's characterization of weak centrality).  We also obtain complete descriptions of more tractable sets of elements, where the corresponding averaging can be done arbitrarily close to the centre. This is achieved through several ``spectral conditions'', involving numerical ranges and tracial states. 

\end{abstract}

\maketitle

\section{Introduction}

Let $A$ be a unital $C^*$-algebra with centre $Z(A)$ and unitary group $\U(A)$. By a \emph{unitary mixing operator} on $A$ we mean a map $\phi \colon A \to A$ of the form
$$
\phi(x)=\sum_{i=1}^n t_i u_i^* x u_i,
$$
where $n$ is a positive integer, $u_1, \ldots, u_n \in \mathcal{U(A)}$ and $t_1, \ldots, t_n$ non-negative real numbers such that $t_1+ \ldots + t_n=1$. The set of all such maps is denoted by $\mathrm{Av}(A,\mathcal{U(A)})$.  Given $a \in A$, we define the \emph{Dixmier set} $D_A(a)$  as the norm-closure of the set $\{\phi(a) : \phi\in\mathrm{Av}(A,\mathcal{U(A)})\}$. If $D_A(a)\cap Z(A)\neq \varnothing$ for all $a \in A$, then $A$ is said to have the Dixmier property. 

It is well-known that if $A$ satisfies the Dixmier property, then $A$ is weakly central, that is, for any pair of maximal ideals  $M_1$ and $M_2$  of $A$, $M_1 \cap Z(A) =M_2 \cap Z(A)$ implies $M_1=M_2$ (see e.g. \cite[p.~275]{Arc78}). It was shown by Haagerup and Zsid\'o in \cite{HZ} that a unital simple $C^*$-algebra satisfies the Dixmier property if and only if it admits at most one tracial state. In particular, weak centrality is not sufficient to guarantee the Dixmier property. A complete generalization of the Haagerup-Zsid\'o theorem was obtained in \cite[Theorem~2.6]{ART}.

%In general, weak centrality is not sufficient to guarantee the Dixmier property. More specifically, it was shown by Haagerup and Zsid\'o in \cite{HZ} that a unital simple $C^*$-algebra satisfies the Dixmier property if and only if it admits at most one tracial state. In particular, weak centrality does not imply the Dixmier property. Recently, a complete generalization of the Haagerup-Zsid\'o theorem was obtained in \cite{ART} 
%(see  \cite[Theorem~2.6]{ART} for the additional necessary conditions). 

In \cite{Mag2}, Magajna gave a characterisation of weak centrality in terms of a more general averaging as follows. By a \emph{unital completely positive elementary operator} on $A$ we mean a map $\phi\colon A \to A$ of the form
\begin{equation}\label{eq:phirep}
\phi(x)=\sum_{i=1}^n a_i^*x a_i,
\end{equation}
where $n$ is a positive integer and $a_1,\ldots, a_n$ elements of $A$ such that $\sum_{i=1}^n a_i^*a_i=1$. The set of all such maps on $A$ is denoted by $\EU(A)$. Given $a \in A$, we define the \emph{Magajna set} $M_A(a)$ as the norm-closure of the set $\{\phi(a) : \phi \in \EU(A)\}$ (i.e. the closed $C^*$-convex hull of $a$ in the sense of \cite{Mag1}). Magajna showed in \cite[Theorem~1.1]{Mag2} that $A$ is weakly central if and only if for any $a \in A$, $M_A(a)\cap Z(A) \neq \varnothing$.

%It was shown by Haagerup and Zsid\'o in \cite{HZ} that a unital simple $C^*$-algebra satisfies the Dixmier property if and only if it admits at most one tracial state. In particular, weak centrality does not imply the Dixmier property. Recently, a complete generalization of the Haagerup-Zsid\'o theorem was obtained in \cite{ART}:
%
%\begin{theorem}[Theorem 2.6 of \cite{ART}] A unital $C^*$-algebra $A$ has the Dixmier property if and only if all of the following hold:
%\begin{itemize}
%\item[(i)] $A$ is weakly central.
%\item[(ii)] Every simple quotient of $A$ has at most one tracial state.
%\item[(iii)] Every extreme tracial state of $A$ factors through some simple quotient.
%\end{itemize}
%\end{theorem}

By a well-known result of Vesterstr\o m \cite{Vest} a unital $C^*$-algebra $A$ is weakly central if and only if it satisfies the centre-quotient property (the CQ-property), that is for any closed two-sided ideal $I$ of $A$, $(Z(A)+I)/I=Z(A/I)$.  
%In \cite{AG}, the authors extended the notion of weak centrality to non-unital $C^*$-algebras (\cite[Definition~3.5]{AG}), obtained a generalization of Vesterstr\o m's theorem (\cite[Theorem~3.16]{AG}) to non-unital case and showed that any $C^*$-algebra $A$ contains a largest weakly central ideal $J_{wc}(A)$. 
Motivated by results in \cite{BG}, where the setting is purely algebraic, the first two authors  defined in  \cite{AG} a local version of the CQ-property as follows. An element $a\in A$ is called a \emph{CQ-element} if for every closed two-sided ideal $I$ of $A$, $a+I \in Z(A/I)$ implies $a \in Z(A)+I$ (\cite[Definition~4.1]{AG}). In \cite[Theorem~4.8]{AG}, a description of the set $\cq(A)$ of all CQ-elements of $A$ was obtained  in terms of those maximal ideals of $A$ which witness the failure of the weak centrality of $A$. Further, it was shown that $\cq(A)$ contains all commutators and products by quasinilpotent elements. However, from both the algebraic and topological viewpoint the set $\cq(A)$ is in general rather strange. It always contains $Z(A)+J_{wc}(A)$ (where $J_{wc}(A)$ is the largest weakly central ideal of $A$, see \cite[Theorem~3.22]{AG}), with equality if and only if $A/J_{wc}(A)$ is abelian. Otherwise it fails dramatically to be a $C^*$-subalgebra of $A$: it is not norm-closed and it is neither closed under addition nor closed under multiplication.

In this paper we study local versions of the Dixmier property and  of weak centrality, the latter motivated by Magajna's characterization in terms of EUCP maps. More precisely, 
given an element $a\in A$, we call $a$ a Magajna element if $M_A(a)\cap Z(A)\neq \varnothing$, and we call it a Dixmier element if $D_A(a)\cap Z(A)\neq \varnothing$. We denote by $\mg(A)$ and $\dx(A)$ the sets of Magajna and Dixmier elements of $A$, respectively. That is, 
\begin{align*}
\mg(A) &:=\{a\in A: M_A(a)\cap Z(A)\neq \varnothing\},\\
\dx(A)  &:= \{a\in A: D_A(a) \cap Z(A)\neq \varnothing\}.
\end{align*}
Obviously, $A$ has the Dixmier property if and only if $\dx(A)=A$, and $A$ is weakly central if and only if $\mg(A)=A$. 
Quasinilpotent elements and self-commutators  are Dixmier (and Magajna) elements. However, a complete description of 
$\dx(A)$ and $\mg(A)$ is in general difficult to obtain. This has led us to also consider the sets
\begin{align*}
\mgbar(A) &=\{a\in A: \ds(M_A(a),Z(A))=0\},\\
\dxbar(A) &= \{a\in A: \ds(D_A(a),Z(A))=0\}.
\end{align*}
These are more tractable sets. In  our two main results, Theorems \ref{thm:sachar} and \ref{thm:condDix}, we characterize membership of a given element in the sets  $\mgbar(A)$ and $\dxbar(A)$ through several ``spectral conditions'', involving numerical ranges and tracial states. Moreover, we show that on selfadjoint elements these spectral conditions also characterize membership in $\mg(A)$ and $\dx(A)$, so that $\mgbar(A)\cap A_{sa}=\mg(A)\cap A_{sa}$ and $\dxbar(A)\cap A_{sa}=\dx(A)\cap A_{sa}$. In general, one has  the inclusions
\[
\xymatrix@R-1pc@C-1pc{
&\dxbar(A) \ar@{}[rd]|-*[@]{\subseteq}&& \\
\dx(A)\ar@{}[ru]|-*[@]{\subseteq}\ar@{}[rd]|-*[@]{\subseteq}&&\mgbar(A)\ar@{}[r]|-*[@]{\subseteq}&\cq(A).\\
&\mg(A)\ar@{}[ru]|-*[@]{\subseteq}&&
}
\]
All these inclusions  may be proper.  A question left unanswered in our investigations is whether the norm closures of $\mg(A)$ and $\dx(A)$ agree with $\mgbar(A)$ and $\dxbar(A)$, respectively (the latter are always closed sets).

If  $A=\mg(A)$, i.e., $A$ is weakly central, then obviously $\mg(A)$ carries a $C^*$-algebra structure. It is thus natural to ask whether $\mg(A)$ can be in general closed under addition or multiplication. As it turns out, imposing either one of these conditions on $\mg(A)$ implies that $A$ is not  far from being weakly central. We show that if  
$\mg(A)$ is either closed under addition or multiplication, then $A/J_{wc}(A)$ is abelian  (Theorem \ref{thm:magclosed}). In the case of $\dx(A)$, the situation is somewhat different: For a large class of unital $C^*$-algebras, which includes those that have a tracial state on every simple quotient, $\dx(A)=\dxbar(A)=Z(A)+\overline{[A,A]}$, and in particular, $\dx(A)$ is closed under addition (Theorem \ref{thm:dixadd}).  On the other hand, if $\dx(A)$ is closed under multiplication, $A$ is not too far from having the Dixmier property, in the sense that $A/J_{dp}(A)$ is abelian, where $J_{dp}(A)$ is the largest ideal of $A$ with the Dixmier property (Theorem \ref{thm:dixmult}).

Although $\mg(A)$ and $\dx(A)$ are not necessarily closed under addition, these sets are always pervasive in the sense that $\overline{\spn(\mg(A))}=Z(A)+\I([A,A])$ (Corollary \ref{cor:spnmag}) and $\overline{\spn(\dx(A))}=Z(A)+\overline{[A,A]}$ (Lemma \ref{zplusc}).

As well as using earlier results on the Dixmier property and weak centrality \cite{AG,ART,Mag1,Mag2,Vest}, the methods of this paper rely on the Dauns-Hofmann theorem \cite{RW}, Michael's selection theorem \cite{Mic} and results on commutators, square zero elements and their lifts \cite{AckPed,marcoux,OP,Pop,RLie}.

\section{Preliminaries}\label{sec:prel}

Let $A$ be a unital $C^*$-algebra. Let $Z(A)$ denote its centre. By an ideal of $A$ we shall always mean a closed two-sided ideal. For any subset $X$ of $A$, we denote by $\I(X)$  the ideal of $A$ generated by $X$. It is well-known that if $I$ is an ideal of $A$, then  $Z(I)=I \cap Z(A)$. As usual, if $x,y \in A$, then $[x,y]$ stands for the commutator $xy-yx$. By $[A,A]$ we denote the linear span of all commutators in $A$. 

By $S(A)$ we denote the sets of all states on $A$. Given $a \in A$ we denote by $W_A(a)$ the (algebraic) numerical range of $a$, that is 
$$
W_A(a)=\{\omega(a) : \omega \in S(A)\}.
$$
It is well-known that $W_A(a)$ is a compact convex set that contains the spectrum of $a$. %, with equality if $a$ is normal \textcolor{red}{(!!!perhaps insert a reference!!!)}. 
We use frequently in what follows that if $I$ is an ideal of $A$, then $W_{A/I}(a+I)\subseteq W_A(a)$.

By $\P(A)$ and $\M(A)$ we respectively denote the sets of all primitive and all maximal ideals of $A$. As usual, we equip $\P(A)$ with the Jacobson topology. As $A$ is unital, both sets $\P(A)$ and $\M(A)$ are compact. For any ideal $I$ of $A$ we write 
$$\M^I(A):=\{M \in \M(A) :  I \subseteq M\}.$$
It is easy to check that the assignment $M \mapsto M/I$ defines a homeomorphism from the set $\M^I(A)$ onto the set $\M(A/I)$.

Let us recall some facts around the Dauns-Hofmann theorem and the complete regularization map (see \cite{AS} for further details).  
For all $P, Q \in \mathrm{Prim}(A)$, define $P \approx Q$ if $P \cap Z(A)= Q \cap Z(A)$. 
By the Dauns-Hofmann theorem \cite[Theorem~A.34]{RW}), there exists a $C^*$-algebra isomorphism $Z(A)\ni z \mapsto \widehat z\in C(\P(A))$ such that 
\[
z+P=\widehat{z}(P)1+P \quad (z\in Z(A),\, P\in \P(A)).
\]
Hence, for all $P,Q \in \P(A)$ we have
$P \approx Q$ if and only if $f(P)=f(Q)$ for all $f \in C(\P(A))$.
Note that $\approx$ is an equivalence relation on $\P(A)$ and the equivalence
classes are closed subsets of $\P(A)$. It follows there is one-to-one correspondence between the quotient set $\mathrm{Prim}(A)/\!\approx$
and a set of ideals of $A$ given by
$[P]_\approx \longleftrightarrow \bigcap [P]_\approx$,
where $[P]_\approx$ denotes the equivalence class of $P\in \P(A)$. The set of ideals obtained in this way is denoted by $\G(A)$, and its
elements are called \textit{Glimm ideals} of $A$. The quotient map $\mathrm{Prim}(A) \to
\mathrm{Glimm}(A)$ given by
\[
P \mapsto \bigcap [P]_\approx
\]
 is known as the \textit{complete regularization map}. 
We equip $\G(A)$ with the quotient topology, which coincides with the complete regularization topology, since $A$ is unital. In this way $\G(A)$ becomes a compact Hausdorff space. In fact, $\G(A)$ is homeomorphic to $\M(Z(A))$ via the assignment 
$\G(A)\ni N \mapsto N \cap Z(A)\in \M(Z(A))$, whose inverse is given by $\M(Z(A))\ni J \mapsto JA \in \G(A)$ ($JA$ is closed by the Hewitt-Cohen factorization theorem).

For $z \in A$ we also write $\widehat{z}$ for the function in $C(\G(A))$ such that 
\[
z+N=\widehat{z}(N)1+N\quad (N\in \G(A)).
\]
Thus, the assignment $Z(A)\ni z \mapsto \widehat{z}\in C(\G(A))$ is a $C^*$-algebra isomorphism.

\section{The set $\mg(A)$}

%Throughout, $A$ is a unital $C^*$-algebra. We denote by $\P(A)$ the space of primitive ideals of $A$. 
%We denote by $\M(A)$ the set of maximal ideals of $A$.  Given a closed (two-sided) ideal $I$, we denote by $\M^I(A)$ the subset of $\M(A)$ of maximal ideals that contain $I$.

%Given $a \in A$ we denote by $W_A(a)$  the (algebraic) numerical range of $a$, that is 
%$$
%W_A(a)=\{\omega(a) :\, \omega \in \S(A)\}.
%$$
%It is well-known that $W_A(a)$ is a compact convex set that contains the spectrum of $a$. %, with equality if $a$ is normal \textcolor{red}{(!!!perhaps insert a reference!!!)}. 
%We use frequently in what follows that if $I$ is an ideal of $A$ then $W_{A/I}(a+I)\subseteq %W_A(a)$.

%We record the following simple observations.
%\begin{remark}\label{rem:Wsub}
%Let $a \in A$.  
%\begin{itemize}
%\item[(a)] For each $\phi\in \EU(A)$ we have $W_A(\phi(a))\subseteq W_A(a)$. This follows directly from the fact that $\omega \circ \phi \in \S(A)$ for any $\omega \in \S(A)$.
%\item[(b)] If $I$ is an ideal of $A$ then $W_{A/I}(a+I)\subseteq W_A(a)$.
%\end{itemize}
%\end{remark}

Recall that for $a\in A$ we define the Magajna set $M_A(a)$ of $a$ as the norm-closure of the set $\{\phi(a) :  \phi \in \EU(A)\}$. Further, we denote by  $\mg(A)$
the set of all $a\in A$ such that $M_A(a)\cap Z(A)\neq \varnothing$.

Our starting point in the investigation of the set $\mg(A)$ is the following theorem of Magajna.

\begin{theorem}[Theorem~4.1 of \cite{Mag1}]\label{thm:Mag4.1} Let $a \in A$. A  normal element $b \in A$ belongs to $M_A(a)$ if and only if $W_{A/P}(b+P)\subseteq W_{A/P}(a+P)$ for each $P\in \P(A)$. 
\end{theorem}

%Let us denote by $\G(A)$ the space of Glimm ideals of $A$.  Recall that (by the Dauns-Hofmann %theorem \cite[Theorem~A.34]{RW}) there exists an isomorphism $Z(A)\ni z \to \widehat z\in C(\G(A))$ %such that 
%\[
%z+N=\widehat{z}(N)1+N
%\]
%for all $z \in Z(A)$ and all $N \in \G(A)$.

Given $a \in A$, let us define the set-valued function $\Psi_a \colon \G(A) \to 2^{\C}$ by 
$$
\Psi_a(N):=\bigcap_{M \in \M^N(A)} W_{A/M}(a+M),
$$
for all $N\in \G(A)$. From Theorem \ref{thm:Mag4.1} we deduce the following proposition:

\begin{proposition}\label{magajnacor}
%Let $A$ be a unital C*-algebra. 
Let $a\in A$ and	$z\in Z(A)$. Then $z\in M_A(a)$ if and only if 
	$\widehat z(N)\in \Psi_a(N)$ for all $N\in \G(A)$, i.e., $ \widehat z$ is a continuous selection of the set-valued map $\Psi_a$.
\end{proposition}

\begin{proof}
Suppose that $\widehat z(N)\in \Psi_a(N)$ for all $N\in \G(A)$, i.e., $\widehat z(N) \in W_{A/M}(a+M)$ for all $N\in \G(A)$ and all $M \in \M^N(A)$. Let  $P\in \P(A)$. Let $N\in \G(A)$ be such that $N\subseteq P$, namely $N=(P\cap Z(A))A$. Since $z+P=\widehat z(N)1+P$, we have that $W_{A/P}(z+P)=\{\widehat z(N)\}$. Choose any $M\in \M(A)$ such that $P\subseteq M$. Then, 
\[
W_{A/P}(z+P)=\{\widehat z(N)\}\subseteq W_{A/M}(a+M) \subseteq W_{A/P}(a+P),
\]
where we have used that $A/M$ is a quotient of $A/P$ in the last inclusion.
It follows from Theorem \ref{thm:Mag4.1} that $z\in M_A(a)$. Conversely, if $z\in M_A(a)$, then again by 
Theorem \ref{thm:Mag4.1} we have $W_{A/P}(z+P)\subseteq W_{A/P}(a+P)$ for all $P\in \P(A)$. Applying this to maximal ideals we get that 
\[
\{\widehat z(N)\}=W_{A/M}(z+M)\subseteq W_{A/M}(a+M)
\] 
for all $N\in \G(A)$ and $M\in \M^N(A)$, as desired.
\end{proof}

\begin{theorem}\label{thm:sachar}
For an element $a \in A$ consider the following conditions:
\begin{itemize}
\item[(i)] $a \in \mg(A)$.
\item[(ii)] $\mathrm{dist}(M_A(a),Z(A))=0$.
\item[(iii)] $\Psi_a(N)\neq \varnothing$ for all $N \in \G(A)$. 
\end{itemize}
Then (i) $\Longrightarrow$ (ii) $\Longleftrightarrow$ (iii). If $a$ is selfadjoint then  (i), (ii) and (iii) are equivalent.
\end{theorem}

Before proving  Theorem \ref{thm:sachar} we establish some preliminary facts. First recall that if $X$ and $Y$ are topological spaces, a set-valued function $\Psi \colon X \to 2^Y$ is said to be \emph{lower semicontinuous} (l.s.c.) if for every open set $U\subseteq Y$, the set
\[
\{x \in X :  \Psi(x)\cap U \neq \varnothing\}
\]
is open in $X$.

\begin{lemma}\label{lem:lsc}
Let $a\in A$. 
\begin{itemize}
\item[{\rm (i)}]
The function $f_a \colon \G(A)\to [0, \infty)$ defined by 
\[
f_a(N):=\inf\{\|a+M\| :\ M \in \M^N(A)\}.
\]
is l.s.c.
\item[\rm{(ii)}]
If $a$ is selfadjoint, then the set-valued function $\Psi_a$ is l.s.c.
\end{itemize}
\end{lemma}
\begin{proof}
(i) Let $s\geq 0$ and 
$$
C_s:= \{N \in \G(A):  f_a(N) \leq s\}.
$$ 
We claim that $C_s$ is closed in $\G(A)$. Indeed, suppose that $(N_\alpha)_\alpha$ is a net in $C_s$ converging to some $N_0 \in \G(A)$. Let $\ep >0$.
For each index $\alpha$ there is $M_\alpha \in \M^{N_{\alpha}}(A)$  such that 
\[
\|a+M_\alpha\| \leq s+\ep.
\] 
As $A$ is unital, $\M(A)$ is compact, so there is a subnet of $(M_\alpha)_\alpha$ convergent to some $M_0\in \M(A)$.
Then, by the continuity of the complete regularization map from $\P(A)$ to $\G(A)$, and the Hausdorffness of $\G(A)$, we conclude that $M_0$ contains $N_0$. By the lower semi-continuity of the norm functions $P \mapsto \|a+P\|$ on $\P(A)$ (see e.g. \cite[Proposition~II.6.5.6 (iii)]{Bla}) we get  $\|a+M_0\| \leq s+\ep$.
Hence $f_a(N_0) \leq s+\ep$. Since $\ep>0$ was arbitrary, $f_a(N_0) \leq s$, as required.

(ii) Given a $C^*$-algebra element $b$, let us denote by $\sp(b)$ its spectrum. Now let $a\in A$ be selfadjoint.  Define two functions $g,h \colon \G(A) \to \R$ by
\begin{align*}
g(N) &:=\inf \{\max \sp(a+M) :  M \in \M^N(A)\},\\
h(N) &:=\sup \{\min \sp(a+M) :  M \in \M^N(A)\}.
\end{align*}
As $\max \sp(a+M)=\|(\|a\|\cdot 1+a)+M\| - \|a\|$, we have that
$$
g(N)=f_{\|a\|1+a}(N)-\|a\|.
$$
 Similarly, $\min \sp(a+M)=\|a\|-\|(\|a\|1-a)+M\|$, from which we deduce that
\begin{align*}
h(N)=-f_{\|a\|1-a}(N)+\|a\|.
\end{align*}
Thus, by part (i), $g$ is lower semi-continuous and $h$ is upper semi-continuous. Since 
the numerical range of a selfadjoint element is the convex hull of its spectrum, 
\[
W_{A/M}(a+M)=[\min \sp(a+M),\max \sp(a+M)]
\]
for all $M\in \M(A)$. We deduce at once that 
\[
\Psi_a(N)=[h(N),g(N)],
\]
where $\Psi_a(N)=\varnothing$ if $h(N)>g(N)$. It follows that $\Psi_a$ is a l.s.c. set-valued function.
\end{proof}

Given a set $C\subseteq \C$ and $\delta>0$, let us denote by $C^{\leq \delta}$ the set of complex numbers whose distance to $C$ is $\leq \delta$. For $a\in A$, $N\in \G(A)$, and $\delta>0$, let us define
\[
W_{N,\delta}(a):=\bigcap_{M\in \mathrm{Max}^N(A)} \Big( W_{A/M}(a+M) \Big)^{\leq\delta}.  
\]	
Now fix $r>0$ and define $\Psi_{a,r}\colon \G(A)\to 2^{\C}$ by
\begin{equation}\label{eq:WNr}
\Psi_{a,r}(N):=\bigcup_{\delta<r} W_{N,\delta}(a).
\end{equation}

\begin{lemma}\label{lem:lscWNr}
Let $a\in A$. For each $r>0$ the map $\Psi_{a,r}$
is lower semi-continuous.
\end{lemma}
\begin{proof}
Let $U\subseteq \C$ be an open set and $N\in \G(A)$ such that $U\cap \Psi_{a,r}(N)\neq \varnothing$.
Let $\omega\in U\cap \Psi_{a,r}(N)$. Observe that this means that $\omega\in W_{N,\delta}(a)$ for some $\delta<r$.
Choose  $\epsilon>0$ such that $\overline{B_\epsilon(\omega)}\subseteq U$.
Suppose, for the sake of contradiction, that there exists a convergent net $N_\lambda\to N$ in $\mathrm{Glimm}(A)$
such that $\overline{B_\epsilon(\omega)}\cap \Psi_{a,r}(N_\lambda)=\varnothing$ for all $\lambda$. Choose $\delta<\delta'<r$.
Then 
\[
\overline{B_\epsilon(\omega)}\cap \bigcap_{M\in \M^{N_\lambda}(A)} \Big( W_{A/M}(a+M) \Big)^{\leq\delta'} =\varnothing,
\]
for all $\lambda$. By Helly's theorem, for each $\lambda$ there exist $M_{\lambda,1},M_{\lambda,2},M_{\lambda,3}$, maximal ideals containing $N_\lambda$, such that 
\begin{equation}\label{varnothing}
\overline{B_\epsilon(\omega)}\cap \bigcap_{i=1,2,3} \Big( W_{A/M_{\lambda,i}}(a+M_{\lambda,i}) \Big)^{\leq\delta'} =\varnothing.
\end{equation}

Pass to subnets if necessary so that $M_{\lambda,i}\to M_i$, with $M_i$ maximal. By the continuity of the complete regularization map and the Hausdorffness of $\G(A)$, we obtain that $N\subseteq M_i$ for $i=1,2,3$. 

Since $\omega\in W_{N,\delta}(a)$, we have that $\omega\in (W_{A/M_i}(a+M_i))^{\leq \delta}$ for $i=1,2,3$. So for each $i$ there exists $\omega_i\in W_{A/M_i}(a+M_i)$ such that $|\omega-\omega_i|\leq \delta$. Now, by the lower semicontinuity
of the map $M\mapsto W_{A/M}(a+M)$ (\cite[Lemma~4.7]{ART}) applied to the three nets
$M_{\lambda,i}\to M_i$, there exists a common index $\lambda_0$ such that 
\[
W_{A/M_{\lambda_0,i}}(a+M_{\lambda_0,i})\cap B_{\delta'-\delta}(\omega_i)\neq \varnothing,
\]
for $i=1,2,3$. This implies that 
\[
\omega\in \Big(W_{A/M_{\lambda_0,i}}(a+M_{\lambda_0,i})\Big)^{\leq \delta'}
\]
for $i=1,2,3$, which contradicts \eqref{varnothing}.
\end{proof}

\begin{lemma}\label{lem:dsMag}
Let $a,b \in A$. Then the following numbers are equal:
\begin{itemize}
\item[(i)] The distance between $M_A(a)$ and $M_A(b)$.
\item[(ii)] The minimum number $r \geq 0$ satisfying 
$$
\ds(W_{A/M}(a+M), W_{A/M}(b+M))\leq r, \quad (M \in \M(A)).
$$
\end{itemize} 
\end{lemma}

\begin{proof}
It is straightforward to show that the number in (i) bounds from above the number in (ii), using that
$W_{A/M}(\phi(a)+M)\subseteq W_{A/M}(a+M)$ for any $\phi\in \EU(A)$. Let us prove that the number in (ii)
bounds the number in (i) from above. We adapt the proof of \cite[Theorem~4.12]{ART} to the present context.

First, let us show that the distance between $M_A(a)$ and $M_A(b)$ is the same as the distance between $M_{A^{**}}(a)$ and $M_{A^{**}}(b)$.  On the one hand, since $M_A(a)\subseteq M_{A^{**}}(a)$ and $M_A(b)\subseteq M_{A^{**}}(b)$,
the distance between the sets $M_A(a)$ and $M_A(b)$ bounds from above the distance between 
the sets   $M_{A^{**}}(a)$ and $M_{A^{**}}(b)$. On the other hand, the opposite inequality follows from a standard  Hahn-Banach/Kaplansky density argument. This argument runs exactly as in the proof of \cite[Lemma~4.1]{ART}, except that unitary mixing operators are replaced by
operators in $\EU(A)$, and the application of the Glimm-Kadison theorem is replaced by the fact that, for $x\in A$, $M_{A^{**}}(x)$ is contained in the $\sigma(A^{**}, A^*)$-closure of $M_A(x)$ (see the proof of \cite[Theorem~4.1]{Mag1}).

Observe now that the number $r$ described in (ii) does not increase if we replace $A$ by $A^{**}$, since for each maximal ideal $M\in \M(A^{**})$ and $a\in A$ we have that  
\[
W_{A/(A\cap M)}(a+A\cap M)=W_{A^{**}/M}(a+M).
\]

In view of the facts observed above, it suffices to prove the desired inequality in the von Neumann algebra $A^{**}$. Let us therefore assume that $A$ is a von Neumann algebra, and in particular, that it is weakly central. Then $\G(A)\cong \M(A)$ via the complete regularization map, and for each $a\in A$ the function $\Psi_a$ is l.s.c. by \cite[Lemma~4.7]{ART} and takes non-empty values (since $\M^N(A)$ is a singleton for all $N\in \G(A)$). 

Let $\epsilon>0$.  Consider the set-valued map 
\[
\G(A)\ni N\mapsto \overline{\Psi_a(N)\cap (\Psi_b(N))^{r+\epsilon}},
\] 
where $(\Psi_b(N))^{r+\epsilon}$ denotes the set of complex numbers whose distance to $\Psi_b(N)$ is $<r+\epsilon$. By \cite[Lemma~4.8]{ART},
this map is l.s.c. and its values are non-empty compact convex subsets of $\C$.  By Michael's selection theorem, it has a continuous selection, which is of the form $N\mapsto \widehat z_1(N)$ for $z_1\in Z(A)$. By Proposition \ref{magajnacor}, $z_1\in M_A(a)$. Consider the set-valued map 
\[
\G(A) \ni N\mapsto \overline{B_{r+2\epsilon}(\widehat z_1(N))\cap \Psi_b(N)}.
\] 
This is again a l.s.c. map whose values are non-empty compact convex subsets of $\C$. Let $z_2\in Z(A)$ be a central element corresponding to a continuous selection of this map. Then $z_2\in M_A(b)$ and $|\widehat z_1(N)-\widehat z_2(N)|\leq r+2\epsilon$ for all $N\in \G(A)$. Hence, $\|z_1-z_2\|\leq r+2\epsilon$. Thus, the distance between $M_A(a)$
and $M_A(b)$ is bounded from above by $r+2\epsilon$ for an arbitrary $\epsilon>0$, whence, also by $r$.
%\textcolor{red}{[Comment: Perhaps add more details in the proof, e.g. for $\ds(M_A(a),M_A(b))=\ds(M_{A^{**}}(a),M_{A^{**}}(b))$,...]}
\end{proof}

\begin{proof}[Proof of Theorem \ref{thm:sachar}] (i) $\Longrightarrow$ (ii). This is trivial.

(ii) $\Longrightarrow$ (iii). Assuming (ii), we get sequences $\phi_n \in \EU(A)$ and $z_n \in Z(A)$ such that
\begin{equation}\label{eq:distzn}
\|\phi_n(a)-z_n\|\to 0.
\end{equation}
Let us fix $N \in \G(A)$.  Set $\lambda_n=\widehat z_n(N)$, so that  $z_n+N=\lambda_n 1+N$ for all $n$. Passing to a subsequence if necessary, assume that $(\lambda_n)_n$
is convergent, with limit  $\tilde \lambda\in \C$. Let $M \in \M^N(A)$ be arbitrary. Since the numerical range of $\phi_n(a)+M$ is contained in the numerical range of $a+M$,  \eqref{eq:distzn}  implies that  
$$
\mbox{dist}(\lambda_n, W_{A/M}(a+M))\to 0.
$$ 
Hence, $\tilde \lambda \in W_{A/M}(a+M)$. As $M \in \M^N(A)$ was arbitrary, it follows $\tilde\lambda \in\Psi_a(N)$ and thus $\Psi_a(N) \neq \varnothing$.

(iii) $\Longrightarrow$ (ii). Let $r>0$. By (iii) we certainly have that $\Psi_{a,r}(N)\neq \varnothing$ for all $N \in \G(A)$. Also, by Lemma \ref{lem:lscWNr}, the set-valued function $\Psi_{a,r}\colon \G(A) \to 2^{\C}$
is lower semi-continuous. So by the Michael selection theorem (\cite[Theorem~3.2]{Mic}), and the Dauns-Hofmann theorem, there is a central element $z \in Z(A)$ such that $\widehat z(N)\in  \Psi_{a,r}(N)$ for all $N \in \G(A)$. This then implies that 
$$
\mathrm{dist}(W_{A/M}(a+M), W_{A/M}(z+M)) \leq r, \quad (M\in \M(A)).
$$
By Lemma \ref{lem:dsMag}, $\textrm{dist}(M_A(a),z)\leq r$. Since $r$ was arbitrary, $\textrm{dist}(M_A(a),Z(A))=0$.

Now assume that $a=a^*$.

(iii) $\Longrightarrow$ (i). Assume that $\Psi_a(N)\neq \varnothing$ for all $N \in \G(A)$.  Since $\Psi_a$
is l.s.c. in this case, by Lemma \ref{lem:lsc}, we can apply the Michael selection theorem directly to $\Psi_a$ to argue the existence of a continuous selection of $\Psi_a$. By Proposition \ref{magajnacor}, this implies that $M_A(a)\cap Z(A)\neq \varnothing$. 
\end{proof}

The next example demonstrates that the equivalent conditions (ii) and (iii) of Theorem \ref{thm:sachar} in general do not imply (i), even if we replace the property of being selfadjoint by normal.

\begin{example}\label{ex:mag3isweaker}
Let $B=\mathcal K(H)+\C p+\C (1-p)$ be the ``Dixmier $C^*$-algebra" (see \cite[NOTE 1,~p.257]{Dix}). Let $A=C([-1,1], M_2(\C))\otimes B$. 
Let $a,b\in C([-1,1],M_2(\C))$ be defined as follows:
\[
a(t):=
\begin{pmatrix}
1 & 0\\
0 & -1
\end{pmatrix}
,\quad
b(t):=
\begin{pmatrix}
	\alpha(t) & 0\\
	0 & \beta(t)
\end{pmatrix},
\]
where $\alpha(t)$ and $\beta(t)$ are curves on the plane such that the interval $[\alpha(t),\beta(t)]$
from $t=-1$ to $t=0$ starts at $[-1,-1+2i]$, remains pinned at $-1$ while rotating till it is flat and equal to $[-1,1]$
at $t=0$. Then from $t=0$ to $t=1$ the interval $[\alpha(t),\beta(t)]$ is pinned at 1, and rotates till it stops at $[1,1+2i]$.
Now define $c\in A$ as
\[
c:=a\otimes p + b\otimes (1-p).
\] 
Identifying $A$ with $C([-1,1],M_2(B))$, $c$ is given by 
\[
c(t)=
\begin{pmatrix}
p + \alpha(t)(1-p) & 0\\
0 & -p+\beta(t)(1-p)
\end{pmatrix}.
\]
Observe that $c$ is a normal element of $A$. 

Let us continue thinking of $A$ as $C([-1,1],M_2(B))$. Its centre is $C([-1,1],\C\cdot 1)$.
Its Glimm ideals are $N_t=\{f\in A :  f(t)=0\}$ for $t\in [-1,1]$. Then $A/N_t\cong M_2(B)$ has two maximal
ideals, $M_2(\mathcal K(H) + \C p)$ and $M_2(\mathcal K(H) + \C (1-p))$. Back in $A$, $N_t$ is contained
in two maximal ideals: 
\begin{align*}
M_{p,t}  &= \{f\in A:f(t)\in M_2(\mathcal K(H) + \C p)\},\\
M_{1-p,t} &= \{f\in A:f(t)\in M_2(\mathcal K(H) + \C (1-p))\}.
\end{align*}
The numerical ranges of the element $c$ with respect to these maximal ideals are
\begin{align*}
W_{A/M_{p,t}}(c+M_{p,t}) &=[\alpha(t),\beta(t)],\\
W_{A/M_{1-p,t}}(c+M_{1-p,t}) &=[-1,1].
\end{align*}
Since 
\[
\Psi_c(N_t)=[-1,1]\cap [\alpha(t),\beta(t)]\neq \varnothing
\] 
for all $t\in [-1,1]$, Theorem \ref{thm:sachar} implies that $\ds(M_A(c),Z(A))=0$. On the other hand,  $t\mapsto [-1,1]\cap [\alpha(t),\beta(t)]$ has no continuous sections. Thus, $c\notin \mg(A)$.
\end{example}

Let $\mgbar(A)$ denote the set of all $a \in A$ such that $\ds(M_A(a),Z(A))=0$.

\begin{proposition}\label{prop:setS}
Let $A$ be a unital $C^*$-algebra.
\begin{itemize}
\item[(i)]  The set $\mgbar(A)$ is a norm-closed subset of $A$ that is contained in $\cq(A)$.
\item[(ii)] %$\mgbar(A)\cap A_{sa}=\mg(A)\cap A_{sa}$. In particular, 
For any $a \in \mgbar(A)$ the real and the imaginary parts of $a$ belong to $\mg(A)\cap A_{sa}$. In particular, $\mgbar(A)\cap A_{sa}=\mg(A)\cap A_{sa}$.
\end{itemize}
\end{proposition}
\begin{proof}
(i) We first show that $\mgbar(A)$ is norm-closed. Assume that $x \in A$ is in the norm-closure of $\mgbar(A)$ and let $\ep>0$. Then there are $a \in \mgbar(A)$, $\phi\in \EU(A)$ and $z \in Z(A)$ such that $\|x-a\|< \ep/2$ and $\|\phi(a)-z\|<\ep/2$. Then, as $\phi$ is contractive, 
$$\|\phi(x)-z\|\leq \|\phi(x)-\phi(a)\|+\|\phi(a)-z\| <\ep$$ 
and thus $x \in \mgbar(A)$. 

We now show that $\mgbar(A) \subseteq \cq(A)$. Assume, for the sake of contradiction, that there is $a \in \mgbar(A) \setminus \cq(A)$. Then by \cite[Theorem~4.8]{AG} there are distinct $M_1,M_2 \in \M(A)$ with $M_1\cap Z(A)=M_2\cap Z(A)$ and distinct scalars $\lambda_1$ and $\lambda_2$ such that $a+M_i=\lambda_i 1+M_i$ ($i=1,2$). As $a \in \mgbar(A)$, for $\ep:=|\lambda_1-\lambda_2|/2>0$ there is $\phi \in \EU(A)$ and $z \in Z(A)$ such that $\|\phi(a)-z\|<\ep$. As  $M_1\cap Z(A)=M_2\cap Z(A)$, there is a scalar $\lambda$ such that $z+M_i=\lambda 1+M_i$ ($i=1,2$).  Thus, 
\[
|\lambda_i-\lambda| = \|(\phi(a)-z)+M_i\|< \ep\quad (i=1,2), 
\]
so that $|\lambda_1-\lambda_2|<2\ep=|\lambda_1-\lambda_2|$;   a contradiction.

(ii) Since operators in $\EU(A)$ commute with the involution, the real and the imaginary parts of any element of $\mgbar(A)$ belong to $\mgbar(A)\cap A_{sa}$. But $\mgbar(A)\cap A_{sa}=\mg(A)\cap A_{sa}$, by the last part of Theorem \ref{thm:sachar}. 
\end{proof}

%\begin{problem} \textcolor{red}{Using \cite[Theorem~4.8]{AG} it is easy to check that 
%\begin{equation}\label{eq:cqsub}
%\{x+iy : \, x,y \in \cq(A) \cap A_{sa}\}\subseteq \cq(A).
%\end{equation}
%On the other hand, by Proposition \ref{prop:setS} we have 
%\begin{equation}\label{eq:mgsub}
%\mgbar(A)\subseteq \{x+iy : \, x,y \in \mg(A) \cap A_{sa}\}.
%\end{equation}
%In general these inclusions are strict. When do we have equalities in \eqref{eq:cqsub} and \eqref{eq:mgsub}?}
%\end{problem}

%A direct argument for the second part: Suppose that $a\in S$ becomes  central in $A/I$. Since $a\in S$,  there exist
%$z_n\in Z(A)$ and $\phi_n\in \EU(A)$ such that $\phi_n(a)-z_n\to 0$. Since $a-\phi_n(a)\in I$,  we have that $z_n+I\to a+I$ in $A/I$. But the image of $Z(A)$ in $A/I$ is closed.

It follows from Proposition \ref{prop:setS} that $\mg(A)\cap A_{\mathrm{sa}}$ is a closed set. Let us show that $\mg(A)$ need not be closed.

\begin{example}\label{ex:magnotclosed}
	Let $A$ and $c\in A$ be as in Example \ref{ex:mag3isweaker}.  Let $\epsilon>0$.
	Define 
	\[
	a_\epsilon(t):=
	\begin{pmatrix}
	1 & \epsilon\\
	0 & -1
	\end{pmatrix},\quad (t\in [-1,1]).
	\] 	
	The numerical range of $a_\epsilon(t)$ is a non-degenerate  elliptical disk with foci at $-1$ and $1$ (\cite[Sec.~210]{halmos}). Call this elliptical disk $E_\epsilon$. Define
	\[
	c_\epsilon:=a_\epsilon\otimes p + b\otimes (1-p)\in A.
	\]
	Clearly, $c_\epsilon\to c$ as $\epsilon\to 0$. Let us argue that $c_\epsilon$ belongs to $\mg(A)$ for all $\epsilon>0$. Indeed, it is clear that $t\mapsto E_\epsilon\cap [\alpha(t),\beta(t)]$ has continuous selections
	for all  $\epsilon>0$ (since $E_\epsilon$ has non-empty interior). By Proposition \ref{magajnacor}, $M_A(c_\epsilon)\cap Z(A)\neq \varnothing$ for all $\epsilon>0$.
\end{example}	

We do not know the answer to the following question:
\begin{problem}\label{prob:mgbar}
Is $\mg(A)$ necessarily norm-dense in $\mgbar(A)$? Equivalently, is $\overline{\mg(A)}=\mgbar(A)$?
\end{problem}

\begin{proposition}\label{sacommag}
	The following classes of elements $a$ satisfy that $0\in M_A(a)$, and thus belong to the set $\mg(A)$:
	\begin{itemize}
	\item[{\rm (i)}]
	self-commutators $[x^*,x]$, with  $x\in A$,
	\item[{\rm (ii)}] 
	$xy$ and $yx$ for all quasinilpotent $x\in A$ and all $y\in A$. 	
	\end{itemize}	
\end{proposition}
\begin{proof}
(i)	Let $a=[x^*,x]$. Then  $0 \in W_A(a)$, since otherwise there would exist $\ep>0$ such that $x^*x\geq \ep 1 + xx^*$ or $xx^* \geq \ep 1+x^*x$, which is seen to be impossible by comparing the norms on both sides. As $a+I=[x^*+I,x+I]$ for any ideal $I$ of $A$, we also have that $0\in W_{A/I}(a+I)$ for any ideal $I$. Thus, $0\in M_A(a)$ by Theorem  \ref{thm:Mag4.1}.

(ii) Let $a=xy$, where $x \in A$ is quasinilpotent and $y \in A$ is arbitrary. Let us argue that $a$ is not invertible: if it were then $x$ would be right invertible. But a quasinilpotent element is neither right nor left invertible. (In general, if $\lambda\in\C$ is a boundary point of the spectrum of an element $a\in A$, then $a-\lambda 1$ is neither left nor right invertible.)
%if $xz=1$ then $x^{n}z=x^{n-1}$, from which
%we deduce that 
%\[
%\|x^n\|\geq \frac{1}{\|z\|}\|x^{n-1}\|\geq \cdots \geq \frac{1}{\|z\|^n}.
%\]
%Hence $\|x^n\|^{\frac1n}\geq 1/\|z\|$ contradicting that $x$ is quasinilpotent.) 
Thus $a$ is not invertible. The same argument applies to the image of $a$ in any quotient of $A$. Thus, $0\in W_{A/I}(a+I)$ for all ideals $I$ of $A$. By Theorem \ref{thm:Mag4.1}, $0\in M_A(a)$, as desired. Similarly, since $x$ is not left invertible, $0 \in M_A(yx)$.	
\end{proof}	

\begin{corollary}\label{cor:spnmag}
We have the following equalities of sets:
\[
\overline{\spn(\mg(A))}=\overline{\spn(\cq(A))}=Z(A)+\I([A,A]).
\]
\end{corollary} 
\begin{proof}
As $\mg(A) \subseteq \cq(A)$, we have that $\overline{\spn(\mg(A))}$ is contained in $\overline{\spn(\cq(A))}$. 

Since the quotient $A/\I([A,A])$ is abelian, $\cq(A)$ is contained in $Z(A)+\I([A,A])$, and since the latter set  is a $C^*$-subalgebra of $A$, we immediately get that $\overline{\spn(\cq(A))}$ is contained in $Z(A)+\I([A,A])$.

It remains to show that $Z(A)+\I([A,A])$ is contained in $\overline{\spn(\mg(A))}$. Since $Z(A) \subseteq \mg(A)$, it suffices to show that $\I([A,A]) \subseteq \overline{\spn(\mg(A))}$. By \cite[Theorem~1.3]{RLie}, $\I([A,A])$ coincides with the set $\overline{[A,A]+[A,A]^2}$. Further, by \cite[Corollary~2.3]{RLie}, $\overline{[A,A]}$ coincides with the closed linear span of all the square zero elements of $A$. Combining these results with Proposition \ref{sacommag} (ii), we conclude that $\I([A,A])$ is contained in $\overline{\spn(\mg(A))}$, as desired.
\end{proof}

By \cite[Proposition~4.5]{AG}, the set $\cq(A)$ contains all commutators. The next example shows that this is no longer true for the sets $\mg(A)$ and $\mgbar(A)$. 

\begin{example}\label{ex:magcom}
	Let $H$ be a separable infinite-dimensional Hilbert space with the orthonormal basis $(\xi_n)$. Define an operator $T \in \mathcal{B}(H)$ by 
	$$
	T\xi_{2n}:=\xi_{2n} \qquad \mbox{and} \qquad T\xi_{2n-1}:=2\xi_{2n-1}, \quad (n \in \N).
	$$ 
	Then $T$ is a (strictly) positive operator with spectrum $\sp(T)=\{1,2\}$, so that $W_{\mathcal{B}(H)}(T)=\mathrm{co}(\sp(T))=[1,2]$.
	It is easy to see that for any scalar $\lambda$, $T-\lambda I$ is not a compact 	operator. Hence, $T$ is a commutator by the main result of \cite{BP}.
	
	Let $\mathcal{C}(H)$ denote the Calkin algebra. Define $A$ to be the $C^*$-algebra of all continuous functions $f\colon  [0,1] \to M_2(\mathcal{C}(H))$ that are diagonal at $1$.
	Consider the constant function $a:=\di(T+\mathcal{K}(H),0)$, where $T$ is as above.
	As $T$ is a commutator, so is $a$. Let $M_1$ and $M_2$ be, respectively, the kernels of $*$-epimorphisms defined by the assignments $f \mapsto f(1)_{11}$ and $f \mapsto f(1)_{22}$. Then $M_1$ and $M_2$ are maximal ideals of $A$ such that  $M_1 \cap Z(A)=M_2 \cap Z(A)$,  
\[
W_{A/M_1}(a+M_1)=W_{\mathcal{C}(H)}(T+\mathcal{K}(H)) \subseteq [1,2],
\]
and $W_{A/M_2}(a+M_2)=\{0\}$. In particular, the sets $W_{A/M_1}(a+M_1)$ and $W_{A/M_2}(a+M_2)$ are disjoint. Thus, $a\notin \mgbar(A)$ by Theorem \ref{thm:sachar}.
\end{example}

\begin{problem}\label{prob:commut}
Is there a simple unital $C^*$-algebra $A$ without  tracial states and such that   $0\in W_A(x)$ for every commutator $x=[a,b]$? See \cite[Example~3.11]{RLie} for a non-simple example.
\end{problem}

%\begin{proposition}
%Let $A$ be a unital $C^*$-algebra and let 
%$$
%D=\{a \in A : \, \ds(M_A(\phi(a)),Z(A))=0 \, \, \forall \phi \in \EU(A)\}.
%$$
%Then $D$ and is a norm-closed subset of $A$ contained in $\mg(A)$. 
%\end{proposition}
%\begin{proof}
%\textcolor{red}{See the proof of Theorem \ref{thm:dixadd}}.
%\end{proof}

In the next theorem we examine various equivalent conditions where $\mg(A)$ is assumed to have additional algebraic structure. 

Recall that, by \cite[Theorem~3.22]{AG}, $A$ contains a largest weakly central ideal $J_{wc}(A)$, and that 
\begin{equation}\label{Jwc}
J_{wc}(A)=\bigcap_{N}\bigcap_{M\in \M^N(A)}M,
\end{equation}
where $N$ runs through the set of Glimm ideals such that $\M^N(A)$ is not a singleton.
(Note: \cite[Theorem~3.22]{AG} also covers non-unital $C^*$-algebras; the description of $J_{wc}(A)$ is slightly different in this case.)
Note that $J_{wc}(A) + \C1$ is weakly central by \cite[Proposition 3.8]{AG} or by an elementary direct sum argument if $J_{wc}(A)$ happens to be unital.
Since the centre of $J_{wc}(A) + \C1$ is contained in $Z(A)$, it follows from Magajna's Theorem \cite[Theorem~1.1]{Mag2} that $J_{wc}(A) + \C1 \subseteq \mg(A)$, and hence $Z(A) + J_{wc}(A) \subseteq \mg(A)$.

\begin{theorem}\label{thm:magclosed}
Let $A$ be a unital $C^*$-algebra. The following conditions are equivalent:
\begin{itemize}
\item[(i)] $\mg(A)=\cq(A)$.
\item[(ii)] $A/J_{wc}(A)$ is abelian.
\item[(iii)] $\mg(A)=\cq(A)=Z(A)+J_{wc}(A).$
\item[(iv)] $\mg(A)$ is closed under addition.
\item[(v)] $\mg(A)$ is closed under multiplication.
\item[(vi)] $\mg(A)$ is closed under EUCP operators.
\end{itemize}
Moreover, under these equivalent conditions $\mg(A)=\mgbar(A)$. 
\end{theorem}
\begin{proof}
Since
\[
Z(A) + J_{wc}(A) \subseteq \mg(A) \subseteq \cq(A),
\]
the equivalence (ii) $\Longleftrightarrow$ (iii) follows from \cite[Theorem~4.12]{AG}.

The implication (iii) $\Longrightarrow$ (i) is obvious. Let us show that (i) $\Longrightarrow$ (ii). Assume that $\mg(A)=\cq(A)$. Then also $\mgbar(A)=\cq(A)$, as $\mg(A)\subseteq \mgbar(A) \subseteq \cq(A)$. In particular, by Proposition \ref{prop:setS}, $\cq(A)$ is norm-closed. By \cite[Theorem~4.12]{AG}, this is equivalent to $A/J_{wc}(A)$ being abelian.

We have so far shown that (i), (ii), and (iii) are equivalent, and that under these conditions $\mg(A)=\mgbar(A)$. Observe  that (iii) implies (iv), (v), and (vi).

Let us show that (iv) $\Longrightarrow$ (ii). Assume, for the sake of contradiction, that $A/J_{wc}(A)$ is non-abelian. By the spectral characterization of $J_{wc}(A)$ \eqref{Jwc}, this implies that there exists $N\in \G(A)$ and distinct maximal ideals $M,M'\in \M^N(A)$ such that $A/M$ is non-abelian. By \cite[Theorems 3.2 and 4.2]{RLie}, the unit of a $C^*$-algebra without 1-dimensional representations  is expressible as a sum 
whose terms are either a  square zero element or a product of two square zero elements. Applied to  $A/M$, this implies that
\[
1 = \sum_{i=1}^m \dot x_i + \sum_{i=1}^n \dot y_i\dot z_i,
\]
where $\dot x_i, \dot y_i, \dot z_i\in A/M$ are square zero elements. Since $M'$ is mapped onto $A/M$ by the quotient map, and square zero elements can be lifted to square zero elements, there exist
square zero  lifts $x_i,y_i,z_i\in M'$ of these elements. Let
\[
a=\sum_{i=1}^m x_i + \sum_{i=1}^n y_iz_i.
\] 
Each term in the sums on the right hand side belongs to $\mg(A)$ by Proposition \ref{sacommag}. As $\mg(A)$ is closed under addition, we conclude that $a\in \mg(A)$. On the other hand,  $a \in M'$ and $a+M =1+M$  imply that  $a \notin \cq(A)$, and in particular $a\notin \mg(A)$. This is the desired contradiction.

%Since $A/M$ is a simple non-abelian $C^*$-algebra, it contains a non-zero $\dot x\in A/M$ such that $\dot x^2=0$. Hence, there exist $\dot d_1,\ldots,\dot d_n\in A/M$ such that
%\[
%1 = \sum_{i=1}^n \dot d_i^*\dot x^*\dot x \dot d_i
%\]
%Choose $M' \in T_A$ such that  $M' \neq M$ and $M \cap Z(A)=M' %\cap Z(A)$. 
%By  \cite[Theorem~6.7]{OP} and the fact that $A/M$ is a quotient of $M'$, we can lift  
%$\dot x$ to a square zero element $x\in M'$. Let $d_i\in A$
%	be lifts of $\dot d_i\in A/M$ for all $i$.	
%	Set $a:=\sum_{i=1}^n d_i^*x^*x d_i$. Now,
%	\[
%	d_i^*x^*xd_i = [(xd_i)^*, xd_i] + xd_id_i^*x^*.
%	\]
%	Both terms on the right are in $\mg(A)$, by Proposition \ref{sacommag}.  By the assumption that $\mg(A)$ is closed under addition, $a\in \mg(A)$. On the other hand,  $a \in M'$ and  $a+M =1+M$ imply that $a \notin \cq(A)$, and in particular $a\notin \mg(A)$. This the desired contradiction.   

Let us show that (v) $\Longrightarrow$ (ii). Let us establish a preliminary fact: Let  $x\in \mg(A)$ and $y,z\in A$ such that $y^2=z^2=0$. 
Choose $\lambda\in \C$ such that $x+\lambda 1$ is invertible. Then 
\begin{align*}
x+y &=(x+\lambda 1)(1 + (x+\lambda 1)^{-1}y) - \lambda 1,\\
x+yz &= (x+\lambda 1)(1 + (x+\lambda 1)^{-1}yz) - \lambda 1.
\end{align*} 
Combining Proposition \ref{sacommag} with the fact that  $\mg(A)$ is closed under multiplication, we deduce that the right hand sides of these equations are  in $\mg(A)$. Thus, $x+y$ and $x+yz$ belong to $\mg(A)$. In particular, $\mg(A)$ contains all finite sums of the form 
\[
\sum_{i=1}^m  x_i + \sum_{i=1}^n  y_i z_i,
\]
where $x_i, y_i, z_i\in A$ are all square zero elements. 

Now, using the same arguments as in the proof of (iv) $\Longrightarrow$ (ii), we conclude that $A/J_{wc}(A)$ must be abelian.

%Now, as in the proof of (iv) $\Longrightarrow$ (ii), assume for the sake of contradiction that 
%there exist $N\in \G(A)$ and distinct $M,M'\in \M^N(A)$ such that $A/M$ is not abelian.
%By \cite[Theorems 3.2 and 4.2]{RLie}, the unit of a $C^*$-algebra without 1-dimensional representations  is expressible as a sum 
%whose terms are either a  square zero element or a product of two square zero elements. Applied to  $A/M$, this implies that
%\[
%1 = \sum_{i=1}^m \dot x_i + \sum_{i=1}^n \dot y_i\dot z_i,
%\]
%where $\dot x_i, \dot y_i, \dot z_i\in A/M$ are square zero elements. Since $M'$ maps onto $A/M$, there exist
%square zero  lifts $x_i,y_i,z_i\in M'$ of these elements. Let
%\[
%a=\sum_{i=1}^m x_i + \sum_{i=1}^n y_iz_i.
%\] 
%By the remark made at the start of the proof, we conclude that $a\in \mg(A)$. On the other hand,  $a \in M'$ and $a+M =1+M$  imply that   $a \notin \cq(A)$, and in particular $a\notin \mg(A)$. This the desired contradiction.   

Let us show that (vi) $\Longrightarrow$ (ii). Again assume, for the sake of contradiction, that there exist $N\in \G(A)$ and distinct $M,M'\in \M^N(A)$ such that $A/M$ is not abelian. Let $\dot x\in A/M$ be a square zero element of norm 1. Since $\dot x$ is not invertible and of norm 1, the numerical range of $\dot x^*\dot x$ is $[0,1]$. By Theorem \ref{thm:Mag4.1}, $1\in M_A(\dot x^*\dot x)$. Thus,
there exist $\dot d_1,\ldots,\dot d_n\in A/M$ such that 
\[
\sum_{i=1}^n (\dot d_i)^* \dot x^*\dot x \dot d_i\geq \frac12(1+M)
\]  
and $\sum_{i=1}^n (\dot d_i)^*\dot d_i=1$. Since $M'$ maps onto $A/M$ by the quotient map, we can lift $\dot x$
to a square zero element of norm one $x\in M'$ (\cite[Proposition 2.8]{AckPed}). Choose also lifts $d_1,\ldots,d_n\in A$ of   $\dot d_1,\ldots,\dot d_n$
such that $\sum_{i=1}^n d_i^*d_i\leq 1$. (This is always possible. Proof: Let $\dot d\in M_n(A/M)$ be the contraction
with first column equal to $(\dot d_1,\ldots, \dot d_n)$ and zeros elsewhere. Let $d\in M_n(A)$ be a lift of $\dot d$
of norm 1. Then the elements on the first column of $d$ are the desired lifts.) Let $c\in A_+$ be such that 
$\sum_{i=1}^n d_i^*d_i+c^2 =1$. Define 
\[
a:=\sum_{i=1}^n d_i^* x^*x d_i + cx^*xc.
\]
We have that $a\in \mg(A)$, by the invariance of $\mg(A)$ under EUCP operators. On the other hand,
$a\in M'$ and $a+M\geq 1/2\cdot 1 + M$, which implies that $\Psi_a(N)=\varnothing$. 
\end{proof}

%Example \ref{ex:magnotclosed} of a $C^*$-algebra such that $\mg(A)$ is not norm closed is also an example where the equivalent conditions of the preceding theorem do not hold.  The following is another  example where these conditions are not satisfied. 

We end this section with another example:

\begin{example}\label{ex:standard3sub}
Let $A$ be the $C^*$-algebra consisting of all functions $a \in C([0,1], M_3(\C))$
such that 
$$a(1)=\begin{pmatrix}
\lambda_{11}(a) & \lambda_{12}(a) & 0 \\
\lambda_{21}(a) & \lambda_{22}(a) & 0 \\
0 & 0 & \mu(a)
\end{pmatrix},$$
for some complex numbers $\lambda_{ij}(a), \mu(a)$, $i,j=1,2$ (see \cite[Example~4.19]{AG}). 
The centre of $A$ is the set of $a$ such that $a(t)\in \C\cdot 1_3$ for all $t\in [0,1]$. The Glimm ideals
are indexed by $[0,1]$: 
\[
N_t=\{a\in A:  a(t)=0\}, \quad (t\in [0,1]). 
\]
Using \eqref{Jwc}, we find that  
\[
J_{wc}(A)=\{a\in A:  a(1)=0\}.
\] 
Observe that, since $A/J_{wc}(A)\cong M_2(\C)\oplus \C$ is non-abelian,
the equivalent conditions of the previous theorem are not met. 
From Theorem \ref{thm:sachar} we obtain a description of $\mgbar(A)$: 
\begin{equation}\label{eq:a1ds}
\mgbar(A) =\Big\{a\in A :
\mu(a)\in W\left(\begin{pmatrix}
\lambda_{11}(a) & \lambda_{12}(a)  \\
\lambda_{21}(a) & \lambda_{22}(a) 
\end{pmatrix}\right)\Big\}. 
\end{equation}
In this case $\mg(A)=\mgbar(A)$. To see this, let $a\in \mgbar(A)$.  By \eqref{eq:a1ds} there is a state $\omega$ on $M_2(\C)$ such that $\omega((\lm_{ij}(a)))=\mu(a)$. Since the map
$M_2(\C)\to M_2(\C)$ given by $b \mapsto \omega(b)1_2$ is completely positive, it belongs to $\EU(M_2(\C))$ by Choi's theorem (see \cite[Theorem~1]{Cho}). Hence, there is  $\phi\in \EU(M_2(\C)\oplus \C)$ such that $\phi(a(1))\in \C 1_3$. By considering the coefficients of $\phi$ as  constant functions in $A$, we may regard $\phi \in \EU(A)$. Then $\phi(a)\in \C \cdot 1_A+J_{wc}(A)$ and this set is contained in $\mg(A)$. Thus, $a\in \mg(A)$. 
%\textcolor{red}{[Comment: Perhaps add some explanation. An EUCP lift of the original $\phi$ is also denoted by $\phi$.]} 

%Let 
%$$a:=\begin{pmatrix}
%1 & 1 & 0 \\
%1 & 1 & 0 \\
%0 & 0 & 0
%\end{pmatrix}$$ 
%(as a constant function). Then $a$ is a positive element of $A$ with $\sp(a)=\{0,2\}$, so  that $W_A(a)=[0,2]$. Further, for each $M \in \M(A)$ we have  $0 \in W_{A/M}(a+M)$, so by Corollary \ref{cor:Magcent} $0 \in M_A(a)$ and thus  $a \in \mg(A)$. On the other hand, consider $\phi \in \EU(A)$ defined by
%$$\phi(x):=e_{11}x e_{11}+e_{22}x e_{22}+e_{33}x e_{33}.$$
%Then $\phi(a)=\di(1,1,0)$, so that $\phi(a)\notin \cq(A)$ (see \cite[Theorem~4.8]{AG}) and thus $\phi(a)\notin \mg(A)$.
%

%Since $A/J_{wc}(A)\cong M_2(\C)\oplus \C$ is non-abelian,
%the equivalent conditions of the previous theorem are not met. 
%To directly  show that $\mg(A)$ is not closed under addition, consider the elements $a:=\di(1,0,0)$ and $b:=\di(0,1,0)$ of $A$. Then $a,b \in \mg(A)$ as for each $M \in \M(A)$ we have $0 \in W_{A/M}(a+M)$ and $0 \in W_{A/M}(b+M)$. On the other hand $a+b=\di(1,1,0)\notin \cq(A)$ (by \cite[Theorem~4.8]{AG}), and consequently $a+b \notin \mg(A)$.  
\end{example}

\section{The set $\dx(A)$}
Throughout this section, we continue to let $A$ denote a unital $C^*$-algebra.  
%We denote by $T(A)$ the tracial stateson $A$. In general, given a $C^*$-algebra $B$ we denote by $T(B)$ the set of tracial states on $B$. 
Recall that given an element $a\in A$, we define its Dixmier set $D_A(a)$ by 
$\overline{\{\phi(a):\phi\colon \mathrm{Av}(A,\mathcal{U(A)})\}}$, where  $\mathrm{Av}(A,\mathcal{U(A)})$ denotes the set of unitary mixing operators on $A$.
Further, we denote by $\dx(A)$ the set of $a\in A$ such that $D_A(a)\cap Z(A)\neq \varnothing$.

%By a \emph{unitary mixing operator} on $A$ we mean a map $\phi :A \to A$ of the form
%$$\phi(x)=\sum_{i=1}^n t_i u_i^* x u_i,$$
%where $n$ is a positive integer, $u_1, \ldots, u_n \in \mathcal{U(A)}$ and $t_1, \ldots, t_n$ non-negative real numbers such that $t_1+ \ldots + t_n=1$. The set of all such maps is denoted by $\mathrm{Av}(A,\mathcal{U(A)})$. 
%%\textcolor{red}{[Comment: Since we already have EUCP, perhaps a similar notation, e.g. UMO?]} 
%Obviously, $\mathrm{Av}(A,\mathcal{U(A)})\subseteq \EU(A)$.

For a $C^*$-algebra $B$ (not necessarily unital), let us denote by $T(B)$ the set of tracial states on $B$.

Define
$$
Y:=\{N \in \G(A):  T(A/N)\neq \varnothing\}.
$$

\begin{lemma}\label{lem:Ycl}
The set $Y$ is a closed subset of $\G(A)$.
\end{lemma}
\begin{proof}
Let $(N_\alpha)_\alpha$ be a net in $Y$ converging to some $N \in \G(A)$. For each $\alpha$, there exists $\tau_\alpha \in T(A)$ such that $\tau_\alpha(N_\alpha)=\{0\}$. Since $T(A)$ is $w^*$-compact, there exists $\tau \in T(A)$ and a subnet $(\tau_{\alpha(\beta)})_\beta$ such that $\tau_{\alpha(\beta)}\underset{\beta}\to  \tau$ in the $w^*$-topology. 

Let us show that $\tau$ vanishes on $N$. 
For each index $\beta$, $\tau_{\alpha(\beta)}$ vanishes on $N_{\alpha(\beta)}\cap Z(A)$, and so $\tau_{\alpha(\beta)}|_{Z(A)}$ is a pure state of $Z(A)$ with kernel $N_{\alpha(\beta)}\cap Z(A)$. Hence, $\tau|_{Z(A)}$ is a pure state of $Z(A)$ with kernel $K$, say. Since $N_{\alpha(\beta)}\underset{\beta}\to N$ in $\G(A)$, we have that \[
N_{\alpha(\beta)} \cap Z(A)\underset{\beta}\to N\cap Z(A)
\] 
in $\M(Z(A))$. On the other hand, 
\[
N_{\alpha(\beta)}\cap Z(A)\underset{\beta}\to K
\] in $\M(Z(A))$. Since $\M(Z(A))$ is Hausdorff, $K=N \cap Z(A)$. From $\tau(N \cap Z(A))=\{0\}$ and the Cauchy-Schwarz inequality for states, we get that
\[
\tau(N)=\tau(A(N \cap Z(A)))=\{0\}. 
\]
Thus, $\tau$ vanishes on  $N$. It follows that $N \in Y$. 
\end{proof}

Recall that given $a\in A$ we denote by $\Psi_a\colon \G(A)\to 2^{\C}$ the function
\[
\Psi_a(N)=\bigcap_{M \in \M^N(A)} W_{A/M}(a+M)\quad (N\in \G(A)).
\]

In the following proposition we make crucial use of \cite[Theorem 4.4]{ART}:
\begin{proposition}\label{dxchar}
Let $a\in A$ and $z\in Z(A)$. 
We have $z\in D_A(a)\cap Z(A)$	if and only if 
\begin{itemize}
\item[{\rm (i)}]
$\widehat z(N)=\tau(a+N)$ for all $N\in Y$ and $\tau\in T(A/N)$,
\item[{\rm (ii)}]
$\widehat z(N)\in \Psi_a(N)$ for all $N\in \G(A)$.	
\end{itemize}	
\end{proposition}	

\begin{proof}
Let $z\in D_A(a)\cap Z(A)$. Let $N\in Y$ and $\tau\in T(A/N)$. Then $z+N=\widehat z(N)1+N$, and evaluating
on $\tau$ we get $\widehat z(N)=\tau(z+N)$. But $\tau$ is constant on $D_{A/N}(a+N)$ 
(since unitary mixing operators preserve the trace) and equal to $\tau(a+N)$. Hence, $\widehat z(N)=\tau(a+N)$.
This proves (i). Condition (ii) follows from the fact that $z$ is in the Magajna set of $a$. 

Conversely, suppose that $z\in Z(A)$ satisfies (i) and (ii). To show that $z\in D_A(a)$, equivalently, that 
$0\in D_A(a-z)$, we rely on  \cite[Theorem~4.4]{ART}. In view of this result, it suffices to show that 
$\tau(a-z)=0$ for all $\tau\in T(A)$ and $0\in W_{A/M}((a-z)+M)$ for all $M\in \M(A)$.
Let $\tau\in T(A)$ be an extreme tracial state. By \cite[Lemma~2.4]{ART} $\tau|_{Z(A)}$ is a pure state on $Z(A)$. Hence, there exists $N\in \G(A)$ such that $\tau$ vanishes on $N\cap Z(A)$. By  the Cauchy-Schwarz inequality for states, we also have  that $\tau(N)=\{0\}$, so that $N \in Y$. By (i) we have 
$$
\tau(a)=\widehat z(N)=\tau(z).
$$
Hence, by the Krein-Milman theorem, $\tau(a)=\tau(z)$ for all $\tau \in T(A)$. 
Let $M \in \M(A)$, and let $N$ be the unique Glimm ideal of $A$ contained in $M$. By (ii), there is a state $\omega$ of $A$ such that $\omega(M)=\{0\}$ and $\omega(a)=\widehat z(N)$. Since $z+M=\widehat z(N)1+M$, $\omega(a-z)=0$ and so
$0 \in W_{A/M}((a-z)+M)$. It follows that $0 \in D_A(a-z)$, as required. 
\end{proof}

\begin{remark}\label{rem:dxchar}
An immediate consequence of  Proposition \ref{dxchar} is that $D_A(a)\cap Z(A)\neq \varnothing$, i.e.
$a\in \dx(A)$, if and only if 
\begin{enumerate}
\item
for each $N\in Y$ all the tracial states of $A/N$ have a common value on $a+N$, call it $f_a(N)$,
\item
there is a continuous selection of $\Psi_a\colon \G(A)\to 2^{\C}$ that agrees with $f_a(N)$ for $N\in Y$.
\end{enumerate}
Indeed, such selections give rise to central elements that satisfy the conditions listed in the proposition. 	
Observe also that if $Y=\varnothing$ (equivalently,  $A$ has no tracial states), then the first condition is vacuously valid; hence $\dx(A)=\mg(A)$ by Proposition \ref{magajnacor} in this case. 
\end{remark}

%\begin{remark}\label{rem:dixnotraces}
%	Suppose that  $A$ has no tracial states (so that by \cite[Theorem~1.1]{ART} $A$ has the Dixmier property  if and only if $A$ is weakly central). Now consider $a \in A$ and $z \in Z(A)$. Then Propositions \ref{magchar} and \ref{dxchar} imply that   $z\in M_A(a)$  if and only if $z \in D_A(a)$. Thus, $\dx(A)=\mg(A)$ in this case. 
%\end{remark}

\begin{theorem}\label{thm:condDix}
Let $A$ be a unital $C^*$-algebra, and let $a \in A$. Consider the following conditions:
\begin{itemize}
\item[(i)] $a \in \dx(A)$.
\item[(ii)] $\ds(D_A(a),Z(A))=0$.
\item[(iii)] 
\begin{itemize}
\item[(a)] There is a function $f_a \colon Y \to \C$ such that 
\begin{itemize}
\item[(a1)] for all $N \in Y$ and $\tau \in T(A/N)$, $f_a(N)=\tau(a+N)$,
\item[(a2)] for all $N \in Y$, $f_a(N)\in \Psi_a(N)$.
\end{itemize}
\item[(b)] For all $N \in \G(A)\setminus Y$, $\Psi_a(N)\neq \varnothing$.
\end{itemize}
\end{itemize}
Then (i) $\Longrightarrow$ (ii) $\Longleftrightarrow$ (iii). Further, if $(iii)$ holds then $f_a$ is unique and it is continuous on $Y$. Finally, if $a$ is selfadjoint, then (i), (ii), and (iii) are equivalent.
\end{theorem}
\begin{proof}
(i) $\Longrightarrow$ (ii). This is immediate.

(ii)  $\Longrightarrow$ (iii). Suppose that $\ds(D_A(a),Z(A))=0$. For each $n \in \N$, there exist a unitary mixing operator $\phi_n$, and $z_n \in Z(A)$, such that 
\begin{equation}\label{eq:smdist}
\|\phi_n(a)-z_n\|< \frac{1}{n}.
\end{equation}
Each $z_n$ gives rise to a continuous function $\widehat z_n \colon \G(A)\to \C$ such that
$z_n+N=\widehat z_n(N)1+N$. Let $N \in Y$ and $\tau \in T(A/N)$. Passing to the quotient $A/N$ in \eqref{eq:smdist} and evaluating on $\tau$, we get
$$
|(\tau(a+N)-\widehat z_n(N)|< \frac{1}{n} \qquad (n \in \N).
$$ 
Hence, $(\widehat z_n(N))_n$ is a convergent sequence with limit $\tau(a+N)$. Thus $\tau(a+N)$ is independent of the choice of $\tau\in T(A/N)$. Define $f_a(N)$ to be this common value. Since $f_a$ is the uniform limit of the sequence of continuous functions $(\widehat z_n|_Y$), $f_a$ is continuous on $Y$. (As shown below, the continuity of $f_a$ can also be derived from the formula $f_a(N)=\tau(a+N)$ ($N\in Y$).)

Let $N \in Y$, $M$ a maximal ideal of $A$ containing $N$, and $\omega$ a state of $A$ such that $\omega(M)=\{0\}$. By \eqref{eq:smdist}, 
$$
|\omega(\phi_n(a))-\widehat z_n(N)|< \frac{1}{n} \qquad (n \in \N).
$$ 
Since $\omega \circ \phi_n$ is a state of $A$ annihilating $M$, 
\begin{equation}\label{eq:dislmn}
\ds(\widehat z_n(N),W_{A/M}(a+M))< \frac1n \qquad (n \in \N).
\end{equation}
But $W_{A/M}(a+M)$ is closed and $\widehat z_n(N) \to f_a(N)$ as $n \to \infty$. Hence $f_a(N)\in W_{A/M}(a+M)$. Since $M$ can vary through all of $\M^N(A)$,
we obtain that $f_a(N)\in \Psi_a(N)$. This proves (iii) (a).

As for (iii) (b), observe that $\ds(M_A(a),Z(A))=0$, as $D_A(a)\subseteq M_A(a)$ (since unitary mixing operators are $\EU$ operators). Thus, $\Psi_a(N)\neq \varnothing$ for all $N\in \G(A)$ by Theorem \ref{thm:sachar}.

%we proceed as in the proof of (ii) $\Longrightarrow$ (iii) of . Namely, let $N \in \G(A) \setminus Y$ and let $\lm_0$ be a cluster point of the bounded sequence $(\widehat z_n(N))_n$. (Note that by \eqref{eq:smdist}  $|\widehat z_n(N)|< \|\phi_n(a)\|+1/n \leq \|a\|+1$ for all $n \in \N$.) Let $M\in \M^N(A)$  and let $\omega$ be a state of $A$ such that $\omega(M)=\{0\}$. As above, we obtain \eqref{eq:dislmn}, so that $\lm_0 \in W_{A/M}(a+M)$ and thus $\lm_0 \in \Psi_a(N)$.

(iii)  $\Longrightarrow$ (ii). Suppose that the element $a$ satisfies (iii). We begin by showing that $f_a$ is continuous on $Y$. Let $(N_\alpha)_\alpha$ be a net in $Y$ convergent to some $N \in Y$ and let $(N_{\alpha(\beta)})_\beta$ be an arbitrary subnet. For each $\beta$, there exists $\tau_{\alpha(\beta)}\in T(A)$ such that $\tau_{\alpha(\beta)}(N_{\alpha(\beta)})=\{0\}$. Since $T(A)$ is $w^*$-compact, there exists $\tau \in T(A)$ and a subnet $(\tau_{\alpha(\beta(\gamma))})_\gamma$ such that $\tau_{\alpha(\beta(\gamma))}\underset{\gamma}\to  \tau$ in the $w^*$-topology. As in the proof of Lemma \ref{lem:Ycl}, we obtain that $\tau(N \cap Z(A))=\{0\}$ and hence $\tau(N)=\{0\}$. We have 
$$
\lim_\gamma f_a(N_{\alpha(\beta(\gamma))})=\lim_\gamma\tau_{\alpha(\beta(\gamma))}(a) = \tau(a)=f_a(N).
$$
It follows that $f_a(N_\alpha)\underset{\alpha}\to f_a(N)$ and so $f_a$ is continuous on $Y$. The uniqueness of $f_a$ is clear from the equation in (iii) (a1).

Let $r >0$. By (iii) (a2), $f_a(N)\in \Psi_{a,r}(N)$ for each $N \in Y$, where $\Psi_{a,r}(N)$ is as in \eqref{eq:WNr}. On the other hand, by (iii) (b), $\Psi_{a,r}(N)$ is  non-empty for each $N \in \G(A) \setminus Y$. As the function $\Psi_{a,r}\colon \G(A) \to 2^{\C}$ is lower semi-continuous, by Michael's selection theorem there exists a continuous function 
$F_a\colon \G(A) \to \C$ such that $F_a(N)\in \Psi_{a,r}(N)$ for all $N\in Y$. Moreover, $F_a$ may be chosen such that $F_a|_Y=f_a$ (\cite[Proposition 1.4]{Mic}). Let $z \in Z(A)$ be the central element such that $\widehat z=F_a$. We now use \cite[Theorem~4.12]{ART} to show that $\ds(D_A(a),z)\leq r$.

Let $\tau \in T(A)$ be an extreme trace. By \cite[Lemma~2.4]{ART}, $\tau|_{Z(A)}$ is a pure state on $Z(A)$. Hence, there exists $N\in \G(A)$ such that $\tau (N\cap Z(A))=0$. By  the Cauchy-Schwarz inequality for states,  $\tau(N)=\{0\}$, so that $N \in Y$. Thus, by (iii) (a) and our construction of $z$, 
$$
\tau(a)=f_a(N)=\tau(z).
$$
Then, by the Krein-Milman theorem, $\tau(a)=\tau(z)$ for all $\tau \in T(A)$. 

Let $M\in \M(A)$ and let $N$ be the unique Glimm ideal contained in $M$. Then $z+N=\widehat z(N)1+N$, and so $W_{A/M}(z+M)=\{\widehat z(N)\}$. Since $\widehat z(N)\in \Psi_{a,r}(N)$, we have that 
$$
\ds(W_{A/M}(a+M),W_{A/M}(z+M))\leq r.
$$
By \cite[Theorem~4.12]{ART}, the distance between $D_A(a)$ and $D_A(z)=\{z\}$ is at most  $r$. Since $r>0$ is arbitrary, $\ds(D_A(a),Z(A))=0$.

Finally suppose that $a=a^*$. Let us show that  (iii) $\Longrightarrow$ (i). Suppose that we have (iii). In this case $\Psi_a$ is l.s.c. by Lemma \ref{lem:lsc}. Thus, by the Michael selection theorem, there exists a continuous selection $F_a\colon \G(A)\to \C$ such that $F_a|_Y=f_a$ and $F_a(N)\in \Psi_a(N)$ for all $N \in \G(A)$.  Let $z\in Z(A)$ be the central element corresponding to $F_a$.  By Proposition \ref{dxchar}, $z\in D_A(a)\cap Z(A)$. Hence $a \in \dx(A)$.
\end{proof}

\begin{proposition}\label{rem:sacommindix}
	The following classes of elements $a$ satisfy that $0\in D_A(a)$, and thus belong to $\dx(A)$.
	\begin{itemize}
	\item[{\rm (i)}]	
	Self-commutators $[x^*,x]$, with $x\in A$.
	\item[{\rm (ii)}]
	Quasinilpotent elements. 
	\end{itemize}
\end{proposition}
\begin{proof}
	(i) Let $a=[x^*,x]$ for some $x \in A$. Then obviously $\tau(a)=0$ for all $\tau \in T(A)$. Furthermore, as seen in the proof of Proposition \ref{sacommag} $0 \in W_{A/I}(a+I)$ for every quotient $A/I$. Thus, we can apply \cite[Theorem~4.4]{ART} to conclude that $0 \in D_A(a)$.
	
	(ii) If $a \in A$ is quasinilpotent then, by the Murphy-West spectral radius formula, for each $\ep>0$ there is an invertible element $x \in A$ such that $\|xax^{-1}\|< \ep$ (see \cite[Proposition~4]{MW}). In particular, $\tau(a)=0$ for all $\tau \in T(A)$. Also, for each ideal $I$ of $A$, $a+I$ is a quasinilpotent element of $A/I$ so that $0 \in W_{A/I}(a+I)$. Invoking again \cite[Theorem~4.4]{ART}, we conclude that $0 \in D_A(a)$.
\end{proof}

\begin{remark}
Unlike $\mg(A)$, the set $\dx(A)$ does not have to contain finite products of nilpotent elements (see Proposition \ref{sacommag}). To demonstrate this, let $A$ be the $C^*$-algebra from Example \ref{ex:standard3sub}. Consider the functions $\tilde e_{ij}$ ($1\leq i,j\leq 2$) in $A$ which are constant and equal to the matrix units 
 $e_{ij}\in M_2(\C)$ ($1\leq i,j\leq 2$). Then $\tilde e_{12}$ and $\tilde e_{21}$ are nilpotent elements of $A$, but  $\tilde e_{12}\tilde e_{21}=\tilde e_{11}\notin \dx(A)$. Indeed, if $N_1=\{a \in A :  a(1)=0\}$, then $N_1 \in \G(A)$ and $A/N_1\cong M_2(\C)  \oplus \C$. The maps $\tau_1, \tau_2 \colon M_2(\C)  \oplus \C \to \C$ given by $\tau_1((a_{ij}),\mu)):=1/2(a_{11}+a_{22})$ and $\tau_2((a_{ij}),\mu)):=\mu$ are tracial states of $M_2(\C)  \oplus \C$ such that $1/2=\tau_1(e_{11})\neq \tau_2(e_{11})=0$. In particular, the condition (iii) (a1) of Theorem \ref{thm:condDix} is not satisfied for $a=\tilde e_{11}$ and $N=N_1$.
\end{remark}

Let us denote by $\dxbar(A)$ the set of $a\in A$ such that $\ds(D_A(a),Z(A))=0$. It is straightforward to check  that $\dxbar(A)$  is a closed set (see the proof of Proposition \ref{prop:setS}). 
 
Note that if $T(A)=\varnothing$,  then $\dxbar(A)=\mgbar(A)$, by Theorems \ref{thm:sachar} and \ref{thm:condDix}.

\begin{theorem}\label{Ytotal} 
Let $A$ be a unital $C^*$-algebra such that for all $N\in \G(A)\backslash Y$, $\M^N(A)$ is a singleton. Then $\dx(A)=\dxbar(A)$.
\end{theorem}

\begin{proof}
For $N\in \G(A)\backslash Y$, say $\M^N(A)=\{M_N\}$.	
Let $a\in \dxbar(A)$. Let $f_a\colon Y\to \C$ denote the function described in Theorem \ref{thm:condDix}.
Define $\tilde\Psi_a \colon \G(A)\to 2^{\C}$ by 
\[
\widetilde\Psi_a(N):=
\begin{cases}
\{f_a(N)\}&\hbox{if }N\in Y,\\
W_{A/M_N}(a+M_N)&\hbox{if }N\in \G(A)\backslash Y.
\end{cases}
\]
Claim: $\widetilde \Psi_a$ is l.s.c. To prove this claim, let us first examine $\widetilde\Psi_a$ on the open set $\G(A)\backslash Y$. Consider the map 
\[
\G(A)\backslash Y\ni N\mapsto M_N\in \M(A).
\]
Let us prove that it is continuous. Let  $N_\alpha\to N$ be a convergent net in $\G(A)\backslash Y$. Then $(M_{N_\alpha})_\alpha$ has convergent subnets, by the compactness of $\M(A)$. Let us show that they all converge to $M_N$: If $M_{N_{\alpha(\beta)}}\underset{\beta}\to M$, with $M\in \M(A)$, then by the continuity of the complete regularization map and the Hausdorffness of $\G(A)$,  $M\in \M^N(A)$. By the assumption that $\M^N(A)$ is a singleton for $N\in \G(A)\backslash Y$, $M=M_N$. This proves the continuity of $N\mapsto M_N$, on the set $\G(A)\backslash Y$.
On the other hand, the set-valued map $\M(A)\ni M\mapsto W_{A/M}(a+M)$ is l.s.c. by \cite[Lemma 4.7]{ART}. Hence, the composition  $N\mapsto W_{A/M_N}(a+M_N)$ is l.s.c. on $\G(A)\backslash Y$.

Let us prove the lower semicontinuity of $\widetilde \Psi_a$ at $N\in Y$. Let $N\in Y$. Let   $\ep>0$. Suppose for the sake of contradiction that there exists a net $(N_\alpha)_\alpha$ in $\G(A)$ such that $N_\alpha\to N$ and $\widetilde\Psi_a(N_\alpha)\cap B_\epsilon(f_a(N))=\varnothing$. Exploiting the continuity of $f_a$ on $Y$, we can find an index $\alpha_0$ such that $N_\alpha\in \G(A)\backslash Y$ for all $\alpha\geq \alpha_0$. Thus, $W(a+M_{N_\alpha})\cap B_\epsilon(f_a(N))=\varnothing$ for all $\alpha\geq \alpha_0$.
Passing to a convergent subnet, and reindexing, let us assume that $M_{N_\alpha}\to M$ for some $M\in \M(A)$. As argued previously, the continuity of the complete regularization map and the Hausdorffness of $\G(A)$ imply that $M\in \M^N(A)$. But $f_a(N)\in W_{A/M}(a+M)$ (by the characterization of elements of $\dxbar(A)$ in Theorem \ref{thm:condDix}). This contradicts the lower semicontinuity of the map  $M\mapsto W_{A/M}(a+M)$ on $\M(A)$ (\cite[Lemma 4.7]{ART}).
This completes the proof of the lower semicontinuity of $\widetilde\Psi_a$.

Let $z\in Z(A)$ be a central element corresponding to a continuous selection of $\widetilde \Psi_a$. Then $z\in D_A(a)$ by Proposition \ref{dxchar}. Thus, $a\in \dx(A)$.
\end{proof}	

\begin{example}\label{ex:dixnotclosed}
Let $B=\mathcal K(H)+\C p+\C (1-p)$ be the ``Dixmier $C^*$-algebra". Let $A=C([-1,1],\mathcal{O}_2)\otimes B$. 
As $T(A)=\varnothing$, we have that $\dx(A)=\mg(A)$, by Remark \ref{rem:dxchar}. Then, arguing as in Example \ref{ex:mag3isweaker}, we can find $a \in \dxbar(A) \setminus \dx(A)$. Moreover, this element belongs to the norm closure of $\dx(A)$ (cf. Example \ref{ex:magnotclosed}).
\end{example}

As with Problem \ref{prob:mgbar}, we do not know the answer to the following question:

\begin{problem} 
Is $\dx(A)$ always dense in $\dxbar(A)$?
\end{problem}

\begin{lemma}\label{zplusc}
The set $Z(A)+\overline{[A,A]}$  contains $\dxbar(A)$ and  is equal to the closed linear span of $\dx(A)$.
\end{lemma}	

\begin{proof}
Let us show first that $Z(A)+\overline{[A,A]}$ is a closed set. Let $a\in \overline{Z(A) + [A,A]}$. Claim: On every quotient $A/N$ by a Glimm ideal $N\in Y$, the element $a$ maps to an element in $\C 1+\overline{[A/N,A/N]}$. Proof: Say $z_n + c_n \to a$, with $z_n\in Z(A)$  and $c_n\in \overline{[A,A]}$ for all $n$. Passing to the quotient $A/N$ we get
\[
\widehat z_n(N)1+(c_n+N)\to a+N.
\]
Let $\tau\in T(A/N)$ (recall that $N \in Y$). Evaluating on $\tau$, we see that  $\widehat z_n(N)\to \tau(a)$. Hence, $c_n+N$ also converges to an element of $\overline{[A/N,A/N]}$, and the claim follows. 

Let $f_a(N)$ denote the common value of all tracial states of $A/N$ on $a+N$, for $N\in Y$. As shown in the proof of Theorem \ref{thm:condDix}, $f_a\colon Y\to \C$ is a continuous function (uniform limit of the functions $\widehat z_n|_Y$).
By the Tietze extension theorem, we can extend $f_a$ to a continuous function on $\G(A)$, and in this way get  $z\in Z(A)$ such that $\widehat z(N)=f_a(N)$ for all $N\in Y$. Now $a - z$ is in the kernel of every extreme trace, since each extreme trace  factors through $A/N$ for some $N\in Y$. By the Krein-Milman theorem, $a-z$ is in the kernel of every trace. It is thus an element in $\overline{[A,A]}$.	 (Indeed, by the Hahn-Banach theorem $\overline{[A,A]}$ agrees with the intersection of the kernels of all bounded functionals that vanish on commutators, and the latter are expressible as linear combinations of tracial states via the Hahn-Jordan decomposition.)

Let us show that $\dxbar(A)$ is contained in $Z(A)+\overline{[A,A]}$. Let $a\in \dxbar(A)$. Let $z_n\in Z(A)$ and $\phi_n\in \mathrm{Av}(A,\mathcal{U(A)})$ be such that $\phi_n(a)-z_n\to 0$. Then $z_n+(a-\phi_n(a))\to a$. The sequence 
$z_n+(a-\phi_n(a))$ belongs to $Z(A)+[A,A]$. Thus, $a$ belongs to $\overline{Z(A)+[A,A]}$, which we have shown agrees with $Z(A)+\overline{[A,A]}$.

Finally, since $\dx(A)$ contains both $Z(A)$ and every self-commutator $[x^*,x]$, $x\in A$, and the latter linearly span $[A,A]$ (\cite[Theorem 2.4]{marcoux}), it follows that the closure of the linear span of $\dx(A)$ is equal to $Z(A)+\overline{[A,A]}$.  
\end{proof}	

\begin{theorem}\label{thm:dixadd}
The following conditions are equivalent:
\begin{itemize}
\item[(i)] $\dx(A)=Z(A)+\overline{[A,A]}$.	
\item[(ii)] $\dx(A)$ is closed under $\mathrm{Av}(A,\mathcal{U(A)})$. 
\item[(iii)] $\dx(A)$ is closed under addition.
\item[(iv)]
\begin{itemize}
\item[(a)] For all $N \in Y$ and $M \in \M^N(A)$, $T(A/M)\neq \varnothing$.
\item[(b)] For all $N \in \G(A)\setminus Y$, $\M^N(A)$ is a singleton set.
\end{itemize}
\end{itemize} 
Moreover, when these equivalent conditions hold, $\dx(A)=\dxbar(A)$. 
\end{theorem}
\begin{proof}
It is clear that (i) implies (ii). 

(ii)  $\Longrightarrow$ (iii). Suppose that (ii) holds. Let $a,b\in \dx(A)$. Let $\epsilon>0$. Choose a unitary mixing operator $\alpha$
and $y_1\in Z(A)$ such that $\|\alpha(a)-y_1\|<\epsilon$. Now choose a unitary mixing operator $\beta$ and $z_1\in Z(A)$ such that 
$\|\beta(\alpha(b))-z_1\|<\epsilon$. Then, setting $\phi:=\beta\circ\alpha$, we have that 
$\|\phi(a)-y_1\|<\epsilon$ and $\|\phi(b)-z_1\|<\epsilon$. Moreover, $\phi(a),\phi(b)\in \dx(A)$. It now follows by a well-known argument of successive averaging that we can find $\psi_n\in  \mathrm{Av}(A,\mathcal{U(A)})$ such that $\psi_n(a)\to y\in Z(A)$ and $\psi_n(b)\to z\in Z(A)$. For completeness, we include the proof. Choose $\phi_1 \in \mathrm{Av}(A,\mathcal{U(A)})$ and $y_1,z_1 \in Z(A)$ such that $\|\phi_1(a)-y_1\|<1/2$ and $\|\phi_1(b)-z_1\|<1/2$. 
Set $a_1:=\phi_1(a)$ and $b_1:=\phi_1(b)$. As $a_1, b_1\in \dx(A)$, there is $\phi_2 \in \mathrm{Av}(A,\mathcal{U(A)})$ and $y_2,z_2 \in Z(A)$ such that $\|\phi_2(a_1)-y_2\|<1/4$,  $\|\phi_2(b_1)-z_2\|<1/4$ . Set $a_2:=\phi_2(a_1)$ and $b_2:=\phi_2(b_1)$. Inductively, we get $\phi_n\in \mathrm{Av}(A,\mathcal{U(A)})$,  $y_n,z_n\in Z(A)$, $a_n\in D_A(a)$, and $b_n\in D_A(b)$,  for $n=2,3,\ldots$, such that 
$$
\|\phi_n(a_{n-1})-y_n\|<\frac{1}{2^{n}}, \qquad \|\phi_n(b_{n-1})-z_n\|<\frac{1}{2^{n}}.
$$
The sequences $(y_n)_n$ and $(z_n)_n$ are Cauchy. Consequently, setting 
\[
y:=\lim_{n\to \infty} y_n, \quad z:=\lim_{n\to \infty} z_n, \quad \mbox{and} \quad \psi_n:=\phi_n \circ \ldots 
\circ \phi_1 \in \mathrm{Av}(A,\mathcal{U(A)}),
\]
 we have that $\psi_n(a)\to y$ and $\psi_n(b)\to z$. Observe now that $\psi_n(a+b)\to y+z$. Hence, $\dx(A)$ is closed under addition, i.e., (iii) holds.

(iii)  $\Longrightarrow$ (iv). Assume that $\dx(A)$ is closed under addition. Let $N \in \G(A)$.
Assume that $N \in Y$, so that $A/N$ has a tracial state. Suppose, for the sake of contradiction,  there exists $M \in \M^N(A)$ such that $T(A/M)=\varnothing$. By Pop's Theorem \cite[Theorem~1]{Pop} we can write the unit of $A/M$ as a finite sum of commutators. By \cite[Theorem 4.2]{RLie}, these commutators are finite sums of square zero elements. 
Thus,  we can write the unit of $A/M$ as a finite sum of square zero elements. Lifting those square zero elements to square zero elements in $A$ and adding, we obtain an element $a\in A$ which is a lift of $1+M\in A/M$ and is a sum of square zero elements (whence in $[A,A]$). Using  Proposition \ref{rem:sacommindix}, and the assumption that $\dx(A)$ is closed under addition, we see that $a \in [A,A] \cap \dx(A)$. Let $z \in D_A(a)\cap Z(A)$. Then $z+N=\lm 1 +N$ and $z+M=\lm 1+M$ for the same scalar $\lm$. Since $\lm 1+N \in D_{A/N}(a+N)$ and $T(A/N)\neq \varnothing$, $\lambda$ is the common value of all tracial states of $A/N$ on $a+N$. Since $a+N \in [A/N,A/N]$,  we have $\lambda=0$. But $\lambda 1 +M \in D_{A/M}(a+M) = \{ 1+M\}$, which gives $\lambda =1$.   This proves (iv)(a).

As for (iv) (b), assume that $N \in \G(A) \setminus Y$, so that $T(A/N)=\varnothing$. Let $\dot{a}$ be an arbitrary element of $A/N$. Combining \cite[Theorem~1]{Pop} and \cite[Theorem 4.2]{RLie}, we can write $\dot{a}$ as a sum of square zero elements in $A/N$. Lifting them back to $A$ as square zero elements, using Proposition \ref{rem:sacommindix} and the assumption that $\dx(A)$ is closed under addition, we get a lift $a \in \dx(A)$ of $\dot{a}$. Let $z \in D_A(a)\cap Z(A)$. Then $z+N \in D_{A/N}(\dot{a})$. But $z+N$ is a scalar in $A/N$. This shows that $A/N$ has the Dixmier property and that $Z(A/N)=\C (1+N)$. Since $A/N$ is weakly central and has trivial centre, it has a unique maximal ideal.

(iv)  $\Longrightarrow$ (i). By assumption, for any $N \in \G(A)\setminus Y$, $N$ is contained in a unique maximal ideal $M_N$ of $A$.

By Theorem \ref{Ytotal}, it suffices to show that $\dxbar(A) = Z(A)+\overline{[A,A]}$. That $\dxbar(A)$ is contained in $Z(A)+\overline{[A,A]}$ has already been established in Lemma \ref{zplusc}. To prove the opposite inclusion, it suffices to show that $\overline{[A,A]}\subseteq \dxbar(A)$. Let $a\in \overline{[A,A]}$. Since traces vanish on commutators, we may define $f_a\colon Y \to \C$ to be identically 0, and for $N \in Y$ and $M\in \M^N(A)$ we obtain $0\in W_{A/M}(a+M)$ by (iv)(a) so that $f_a(N) = 0 \in \Psi_a(N)$.
 On the other hand,  for $N\in \G(A)\backslash Y$ we have that $\Psi_a(N)=W_{A/M_N}(a+M_N)\neq \varnothing$. It follows from Theorem \ref{thm:condDix} that $a\in \dxbar(A)$. 
%Having shown that $\dxbar(A) = Z(A)+\overline{[A,A]}$, observe this set is invariant under unitary mixing operators. It follows by the successive averagings technique that we used above that $\dx(A)=\dxbar(A)$, which proves (i). 
%We can also appeal to the previously established result that if $\M^N(A)$ is a singlton for $N\in \G(A)\backslash Y$, then both sets agree.
\end{proof}

%\begin{remark}\label{rem:commut}
%	Let $A$ be a unital $C^*$-algebra. As in \cite{AG} by $T_A$ we denote the set of all $M \in \M(A)$ for which exists $M' \in \M(A)$ such that $M\neq M'$ and $M \cap Z(A)=M' \cap Z(A)$. 
%	\begin{itemize} 
%		\item[(a)] If all simple quotients $A/M$ have at least one tracial state, then $\overline{[A,A]} \subseteq \dx(A)$. Indeed, if $a \in \overline{[A,A]}$ then $a$ satisfies the conditions of \cite[Theorem~4.4]{ART}, so that $0 \in D_A(a)$ and thus $a \in \dx(A)$. 
%		\item[(b)] If for all $M \in T_A$, $A/M$ admits a tracial state then $\overline{[A,A]} \subseteq S$ (where, as before, $S=\{a \in A: \, \mathrm{dist}(M_A(a),Z(A))=0\}$). Indeed, if $a \in \overline{[A,A]}$ then for all $M \in T_A$ we have $0 \in W_{A/M}(a+M)$. In particular $\Psi_a(N) \neq \emptyset$ for all $N \in \G(A)$ and thus, by Theorem \ref{thm:sachar}, $a \in S$. 
%		\item[(c)] If there is $M \in T_A$ such that $A/M$ does not admit a tracial state, then by \cite[Theorem~4.10]{AG} $[A,A]\nsubseteq \cq(A)$.
%	\end{itemize} 
%\end{remark}
%

\begin{remark}
Theorem \ref{thm:dixadd} in particular applies to all unital $C^*$-algebras such that any non-zero simple quotient has at least one tracial state. These include all unital postliminal $C^*$-algebras, as each simple quotient of such an algebra is isomorphic to some full matrix algebra $M_n(\C)$, % and thus has a (unique) tracial state, 
and all  unital AF $C^*$-algebras, as in this case the quotients are again unital and AF, and thus have tracial states. Observe that among such $C^*$-algebras, $\mg(A)$ may well fail
to be closed under addition (e.g., the postliminal $C^*$-algebras from Examples \ref{ex:mag3isweaker} and \ref{ex:standard3sub}).
\end{remark}

\begin{example}
Let $A$ be a unital $C^*$-algebra with a unique faithful trace $\tau$ and a unique maximal ideal $M$ such that $A/M$ is traceless.
(See \cite[Section 2]{ART} for concrete examples of $C^*$-algebras with these properties.) Since $A$ is weakly central, $A=\mg(A)$. On the other hand,
\[
\dx(A)=\{a\in A:\tau(a)\in W_{A/M}(a+M)\}.
\] 
This set is not closed under addition (Theorem \ref{thm:dixadd} (iv) (a) fails) and maps onto $A/M$.  Moreover,  $\dx(A)=\dxbar(A)$ in this case, by Theorem \ref{Ytotal}. This description of $\dx(A)$ can be considered in terms of convexity.
Viewing $S(A/M)$ as a $w^*$-compact convex subset of  $A^*$ that does not contain the tracial state $\tau$,  a suitable form of the Hahn-Banach theorem tells us that we can separate by an element of the dual space $A$  (for example, any element $a\in M$ such that $\tau(a)\neq 0$).
Then the above description of $\dx(A)$ tells us that the elements $a\in A$ that can witness this separation are precisely those in the complement of $\dx(A)$. 	
\end{example}	

%In general, $\dx(A)$ is not closed under addition.
%\begin{example}
%Let $A$ be the $C^*$-algebra from Example \ref{ex:magcom}. Let $p \in \BB(\H)$ be any projection with infinite-dimensional kernel and image and set $a:=\di(p+\KK(\H),0)$ and $b:=\di(1-p+\KK(\H),0)$. Then $a,b$ are non-trivial projections in $A$ and for all $M \in \M(A)$ we have $0 \in W_{A/M}(a+M)$ and $0 \in W_{A/M}(b+M)$, so that by Remark \ref{rem:dixnotraces} $a,b \in \mg(A)=\dx(A)$. But $a+b=\di(1+\KK(\H),0)\notin \cq(A)$.
%\end{example}

By \cite[Proposition 2.1.10]{ArcTh} any unital $C^*$-algebra $A$ contains a largest ideal $J_{dp}(A)$ with the Dixmier property (in the sense that $J_{dp}(A)+\C 1_A$ has the Dixmier property). Since $Z(J)=J\cap Z(A)$ for any ideal $J$ of $A$, we have that $Z(A)+J_{dp}(A)\subseteq \dx(A)$.  Note also that if $J$ is an ideal of $A$ that is contained in $\dx(A)$, then $J$ has the Dixmier property and hence is contained in $J_{dp}(A)$. This follows from the fact that $D_{J+\C 1_A}(a)=D_A(a)$ for all $a\in J$  (see \cite[Remark~2.6]{Arc77}).

In the next theorem we describe $J_{dp}(A)$ more explicitly. Before doing so, we introduce some notation. Consider the set $X\subseteq \G(A)$ defined as
\[
X:=\{N\in \G(A):\hbox{$A/N$ has the Dixmier property and trivial centre}\}.
\]
By \cite[Corollary 2.10]{ART}, $N\in X$ if and only if   $\M^N(A)$ is a singleton $\{M_N\}$,  $A/N$ has at most one tracial state,  and if $A/N$ does have a tracial state then it factors through $A/M_N$. Now, for $N\in \G(A)\backslash X$, define 
\[
I_N:=\left(\bigcap_{M\in \M^N(A)} M \right)\cap \bigcap_{\tau\in T(A/N)} I_\tau,
\]
where, for $\tau\in T(A/N)$,
$$I_\tau: =\{a\in A:  \tau(a^*a+N)=0\}.$$

We adopt the usual convention that the intersection of an empty set of ideals of $A$ is equal to $A$ itself.

\begin{theorem}
We have 
\[
J_{dp}(A)=\bigcap_{N\in \G(A)\backslash X} I_N,
\]
with $X\subseteq \G(A)$ and $I_N$ for $N\in \G(A)\backslash X$ as defined above.
\end{theorem}

\begin{proof}
Call $J$ the intersection of the ideals on the right hand side of the equation. Let us show that  $J_{dp}(A)$ is contained in $J$.
Since $J_{dp}(A)+\C 1_A$ has the Dixmier property, it is weakly central, and hence so is $J_{dp}(A)$ (by  \cite[Remark~3.10]{AG}). Hence $J_{dp}(A)\subseteq J_{wc}(A)$. So, 
by the spectral characterization of $J_{wc}(A)$ \eqref{Jwc}, we have that $J_{dp}(A) \subseteq \bigcap_{M\in \M^N(A)} M$ for each $N\in \G(A)\backslash X$ such that $\M^N(A)$ is not a singleton.

Suppose that $N\in \G(A)\backslash X$ is such that $T(A/N)\neq \varnothing$. Let $\tau\in T(A/N)$, and suppose for the sake of contradiction that $J_{dp}(A)$ is not contained in $I_{\tau}$. Then there is a positive element $b\in J_{dp}(A)$ such that $\tau(b+N)>0$. Since $b\in \dx(A)$, $b+N$ can be averaged in $A/N$ by unitary mixing operators to a scalar multiple of the identity. Hence, $\tau(b+N)1_{A/N}\in D_{A/N}(b+N)$. It follows that the ideal $J_{dp}(A)$ maps onto 
$A/N$ under the quotient map $A \to A/N$. Since $J_{dp}(A)+\C 1_A$ has the Dixmier property, with centre contained in $Z(A)$, $A/N$ has the Dixmier property and trivial centre. Thus $N\in X$, contradicting our choice of $N$.

It remains to deal with the case where $N\in \G(A)\setminus X$ and $N$ is contained in a unique maximal ideal $M_N$ of $A$. Since $N\notin X$, $T(A/N)\neq\varnothing$. Let $\tau\in T(A/N)$. Then $I_\tau$ is a proper ideal of $A$ containing $N$, and hence it is contained in $M_N$. Since $J_{dp}(A)\subseteq I_{\tau}$ (by the previous paragraph), $J_{dp}(A)\subseteq M_N$ as required.

%not every trace of $A/N$ factors through a maximal ideal. Let %$N$ be a Glimm ideal as in Suppose that $b$ maps outside $I_%%\tau$. Then $b'=b^*b$, also in $J_{dp}$,  maps to an element %$b'+N$ such that $\tau(b'+N)\neq 0$. Since $b'\in \dx(A)$, it %averages to the center. Its image in $A/N$ averages to a %nonzero scalar multiple of the identity. We deduce that
%$J_{dp}$ maps onto $A/N$. But a quotient of a C*-algebra with %the Dixmier property has the Dixmier property. This %contradicts that $N\notin X$. 

For the reverse inclusion, let $a\in J$. We claim that $a\in \dx(A)$. To prove this, we apply Theorem \ref{thm:condDix}. Observe that  $\Psi_a(N) =\{0\}$ for all $N\in\G(A)$ for which $\M^N(A)$ is not a singleton. Hence, $\Psi_a(N)\neq \varnothing$ for all $N\in \G(A)$. Let $N$ be a Glimm ideal in the set $Y$ (i.e., such that $T(A/N)$ is non-empty). Suppose first that $N\in Y\setminus X$. Then, since $a\in I_\tau$ for all $\tau\in T(A/N)$, $a+N$ belongs to the kernel of all traces in $T(A/N)$. Set $f_a(N)=0$. Suppose on the other hand that $N\in Y\cap X$. Define $f_a(N)= \tau_N(a+N)$, where $\tau_N$ is the unique tracial state of $A/N$. Now $f_a$ satisfies the conditions (a1) and (a2) of Theorem \ref{thm:condDix}. Hence, $a\in \dxbar(A)$. Since we have proved this for every element of $J$ (a set invariant under unitary mixing operators), it follows by the method of successive averaging already employed in the proof of Theorem \ref{thm:dixadd} that $J\subseteq \dx(A)$.  As noted earlier, this implies that $J$ has the Dixmier property, and so $J\subseteq J_{dp}(A)$. 
\end{proof}

\begin{theorem}\label{thm:dixmult}
The following conditions are equivalent:
\begin{itemize}
\item[{\rm (i)}]
$A/J_{dp}(A)$ is abelian.
\item[{\rm (ii)}]
$\dx(A)=Z(A)+J_{dp}(A)$,
\item[{\rm (iii)}]
$\dx(A)$ is closed under multiplication.
\end{itemize}		
Moreover, under these equivalent conditions $J_{dp}(A)=J_{wc}(A)$, so that
\begin{equation}\label{eq:allequal}
\dx(A)=\mg(A) = Z(A) + J_{dp}(A) = Z(A) + \overline{[A,A]}.
\end{equation}

\end{theorem}	

\begin{proof}
(i) $\Longrightarrow$ (ii). Suppose that $A/J_{dp}(A)$ is abelian and let $a\in \dx(A)$. Then $a\in \cq(A)$ and so, since the image of $a$ in $A/J_{dp}$ is central, $a\in Z(A) + J_{dp}(A)$. 
% Since $\dx(A)$	is contained in $\cq(A)$, and every %element of $A$ is central under the quotient by $J_{dp}$, the %image of  $\dx(A)$ in $A/J_{dp}$ has central lifts. 
Thus $\dx(A)\subseteq Z(A)+J_{dp}(A)$ and the reverse 
inclusion always holds as previously noted.
% since elementts in $Z(A)+J_{dp}$ can be averaged to the %center.  

Clearly (ii) $\Longrightarrow$ (iii).

(iii) $\Longrightarrow$ (i). Let us show first that $\dx(A)$ is closed under addition.	
Let $x\in A$ be a square zero element and let $y\in \dx(A)$. Then
	\[
	 x+y = (1+x)(1 + (1-x)y) - 1.
	\]
Using that $\dx(A)$ is closed under multiplication (and translations by scalar multiples of 1) we deduce that the right hand side is in $\dx(A)$. Hence, $x+y$ is in $\dx(A)$. It follows that $\dx(A)$ contains all finite sums of square zero elements. This enables us to check the spectral characterization in Theorem \ref{thm:dixadd} (iv) for when $\dx(A)$ is closed under addition. Indeed, in the proof of Theorem \ref{thm:dixadd} (iii) $\Longrightarrow$ (iv), all we used from (iii) was that $\dx(A)$ contains all finite sums of square zero elements.

Since $\dx(A)$ is closed under  addition,  $\dx(A)=Z(A)+\overline{[A,A]}$ by Theorem \ref{thm:dixadd}. 
In particular, $\dx(A)$ is norm-closed and contains $\overline{[A,A]}$. Hence, it contains the $C^*$-algebra generated by 
$\overline{[A,A]}$. But this $C^*$-algebra agrees with the ideal generated by $\overline{[A,A]}$ (\cite[Theorem 1.3]{RLie}). Thus $\mathrm{Ideal}([A,A])\subseteq \dx(A)$.
% i.e., the elements of $\mathrm{Ideal}([A,A])$ can be %averaged to the center by means of unitary mixing operators.
 Hence, as noted earlier, this ideal is contained in $J_{dp}(A)$. Thus $A/J_{dp}(A)$ is abelian, as desired.
 
Finally, assume that $A/J_{dp}(A)$ is abelian. Let $B:=J_{wc}(A)+\C1_A$, so that $B$ is a unital weakly central $C^*$-algebra. Since $J_{dp}(B)$ contains $J_{dp}(A)$,  $B/J_{dp}(B)$ is abelian. As weak centrality is equivalent to the CQ-property (\cite{Vest}), this implies $B=Z(B)+J_{dp}(B)$. Hence $B$ has the Dixmier property, and so $J_{wc}(A) = J_{dp}(A)$.
%Hence, $J_{wc}(A)\subseteq \dx(B)\subseteq \dx(A)$, which shows that $J_{wc}(A)$ has the Dixmier property. Thus $J_{dp}(A)=J_{wc}(A)$. 
The equality \eqref{eq:allequal} now follows directly from (ii), Theorem \ref{thm:dixadd}, and Theorem \ref{thm:magclosed}. 
\end{proof}	

\begin{remark}
Note that the equivalent conditions of Theorem \ref{thm:dixmult} are satisfied for the ``Dixmier $C^*$-algebra" and also for all $2$-subhomogeneous $C^*$-algebras (see \cite[Corollary~4.15]{AG}). On the other hand, if $A$ is the $C^*$-algebra from Example \ref{ex:standard3sub}, then $A/J_{wc}(A)$ is non-abelian, and thus $\dx(A)$ fails to be closed under multiplication.
\end{remark}
	
\begin{corollary}
Suppose that either $\mg(A)$ or $\dx(A)$ is closed under addition. Then $\mg(A)=\dx(A)$ if and only if $A$ satisfies the equivalent conditions of Theorem \ref{thm:dixmult}.
\end{corollary}

\begin{proof}  If $\mg(A)=\dx(A)$ then $\dx(A)$ is closed under multiplication by Theorem \ref{thm:magclosed}. Conversely, if $A$ satisfies the equivalent conditions of Theorem \ref{thm:dixmult} then $\mg(A)=\dx(A)$ as shown above.
\end{proof}

\end{document}